\documentclass[11pt]{article}
\usepackage[a4paper,margin=1in]{geometry}
\usepackage{amsmath,amssymb,amsthm,mathtools,bm,mathrsfs}
\usepackage{enumitem}
\usepackage{xcolor}
\usepackage{algorithm}
\usepackage{algpseudocode}
\usepackage[numbers,sort&compress]{natbib}
\usepackage[colorlinks=true,linkcolor=blue,citecolor=blue,urlcolor=blue]{hyperref}

\providecommand{\headers}[2]{}
\providecommand{\email}[1]{\href{mailto:#1}{#1}}
\newenvironment{keywords}
  {\par\smallskip\noindent\textbf{Keywords: }}
  {\par\smallskip}
\newenvironment{MSCcodes}
  {\par\smallskip\noindent\textbf{2020 Mathematics Subject Classification: }}
  {\par\smallskip}

\numberwithin{equation}{section}
\numberwithin{algorithm}{section}
\newcommand{\R}{\mathbb{R}}
\newcommand{\C}{\mathbb{C}}

\newcommand{\E}{\mathbb{E}}

\newcommand{\A}{\mathcal{A}}
\newcommand{\norm}[1]{\lVert #1\rVert}

\theoremstyle{plain}
\newtheorem{theorem}{Theorem}[section]
\newtheorem{proposition}[theorem]{Proposition}
\newtheorem{lemma}[theorem]{Lemma}
\newtheorem{corollary}[theorem]{Corollary}
\newtheorem{definition}{Definition}[section]
\newtheorem{remark}{Remark}[section]
\newtheorem{assumption}{Assumption}[section]

\newtheoremstyle{algstyle}{3pt}{3pt}{\normalfont}{}{\bfseries}{.}{.5em}{}
\theoremstyle{algstyle}
\theoremstyle{plain}

\headers{Parabolic equations in spectral Barron spaces}{S. Ji and X. Wang}
\title{Variable-Coefficient Parabolic Equations and Finite-Horizon Hamilton--Jacobi--Bellman Equations in Augmented Spectral Barron Spaces}
\author{
Shaolin Ji\thanks{
Zhongtai Securities Institute for Financial Studies,
Shandong University, Jinan, Shandong 250100, PR China
(\email{jsl@sdu.edu.cn}).
}
\and
Xianrui Wang\thanks{
Corresponding author.
Zhongtai Securities Institute for Financial Studies,
Shandong University, Jinan, Shandong 250100, PR China
(\email{202511852@mail.sdu.edu.cn}).
}
}

\begin{document}

\maketitle
\begin{abstract}
We establish a whole-space solution framework in augmented spectral Barron spaces for uniformly elliptic parabolic equations with state-dependent principal coefficients. A frozen-symbol parametrix yields bounded Green and terminal operators and a one-derivative smoothing estimate on arbitrary finite horizons, without requiring the spatial variation of the principal coefficient to be perturbatively small. We then apply this linear framework to a finite-horizon Hamilton--Jacobi--Bellman equation. A semi-explicit gradient iteration converges on sufficiently short horizons at a Gamma-factorial rate in a whole-space augmented Barron norm, and its limit is identified with the stochastic-control value function. Finally, temporal Jackson approximation and spatial spectral Barron approximation give joint shallow cosine-network approximations in space and time for the value function and the optimal feedback, with quantitative error and neuron-count bounds. Taken together, these results connect variable-coefficient parabolic well-posedness and smoothing in augmented spectral Barron spaces with nonlinear HJB solvability and quantitative neural-network approximation.

\end{abstract}

\section{Introduction}
High-dimensional parabolic partial differential equations arise naturally in stochastic control, quantitative finance, and diffusion models. Their numerical treatment is challenging because conventional mesh-based discretizations typically become increasingly expensive as the spatial dimension grows. This difficulty has motivated neural-network-based approaches to high-dimensional PDEs, including physics-informed neural networks \cite{raissi2019} and Deep BSDE methods \cite{han2018}. From an analytical perspective, however, a complementary question is fundamental: in which function spaces can solutions of high-dimensional PDEs be shown to admit quantitative neural-network approximations?

Barron spaces provide a natural framework for addressing this question. Barron's approximation theorem \cite{barron1993} shows that functions with suitable Fourier complexity admit shallow-network approximations with the characteristic \(N^{-1/2}\) rate, whose width-dependent exponent does not explicitly deteriorate with the ambient dimension. Spectral Barron spaces formulate this complexity through weighted Fourier-\(L^1\) norms and are therefore particularly well suited to PDE analysis. E and Wojtowytsch \cite{e2021barron} demonstrated how Barron regularity can be propagated through explicit solution representations for several whole-space PDEs. Subsequent spectral Barron solution theories include the static Schrödinger theory of Chen, Lu, Lu, and Zhou \cite{chen2023} and the static Hamilton--Jacobi--Bellman theory of Feng and Lu \cite{feng2026}. In these settings, the principal second-order operator retains a constant-coefficient Fourier structure. Motivated by this program, we study the finite-horizon HJB equation
\begin{equation}
\partial_t V(t;x)+\inf_{u}\left\{L^u V(t;x)+\ell(t;x)+u^{\top}Ru\right\}-\gamma V(t,x)=0,\qquad V(T;x)=\varphi(x),\label{eq:hjb}
\end{equation}
where the controlled second-order operator $L^u$ is given by
\begin{equation}
L^u v(t;x)=\bigl(f(t;x)+g(t;x)u\bigr)\cdot\nabla_x v(t;x)+\frac12\,\sigma(t;x)\sigma(t;x)^{\top}:D_x^2 v(t;x).\label{eq:controlled-operator}
\end{equation}
Writing \(a(t,x):=\sigma(t,x)\sigma(t,x)^\top\) is uniformly elliptic and genuinely depends on the state variable. The principal analytical difficulty is therefore already present at the linear level. If \(a\) is independent of \(x\), the principal parabolic operator is diagonalized by the spatial Fourier transform and can be treated frequency by frequency through an explicit multiplier. Once \(a=a(t,x)\), multiplication by the coefficient becomes convolution in Fourier space, so this diagonal structure is lost at the level of the principal second-order operator. The constant-coefficient Fourier mechanisms underlying the preceding spectral Barron theories therefore no longer apply directly.

Recent developments make the remaining difficulty more precise. Chen, Huang, Yang, and Zhou \cite{chen2026a} establish spectral Barron regularity for second-order elliptic equations with variable leading-order coefficients, but their argument is stationary and perturbative, requiring the variable part of the principal coefficient to be sufficiently small relative to a constant elliptic background. It therefore does not provide the time-dependent Green and smoothing theory needed for the parabolic problem considered here. Choi et al. \cite{choi2026} develop a finite-time theory for fractional parabolic equations in anisotropic space--time spectral Barron spaces, but the principal fractional heat operator remains translation invariant and the spatially dependent coefficients enter through lower-order drift and potential terms. Hence the principal operator can still be treated through its Fourier multiplier, whereas the state-dependent diffusion coefficient in \eqref{eq:controlled-operator} destroys this diagonal Fourier structure. Choulli, Lu, and Takase \cite{choulli2026heat} study heat equations in spectral Barron spaces and also treat variable-coefficient diffusion operators, but the variable principal part is handled perturbatively around a constant coefficient matrix.

Thus, the existing results do not directly furnish a whole-space spectral Barron solution framework for the time-dependent nondivergence operator
\[
\partial_t+\frac12 a(t,x):D_x^2
\]
when the principal coefficient depends genuinely on the state variable and its spatial variation is not assumed to be perturbatively small. This is the linear analytical problem that must be resolved before the Barron-space HJB analysis can be carried out.

Our first step is to establish a whole-space solution theory for the variable-coefficient linear parabolic equation in augmented spectral Barron spaces. The appropriate scale is
\[
\mathcal A^s=\mathbb C\oplus\mathcal B^s,
\]
whose constant component accommodates the nondegenerate elliptic background, while the \(\mathcal B^s\)-component describes the spatial variation. We show that, for uniformly elliptic state-dependent principal coefficients in this class, the linear terminal-value problem admits a unique solution
\[
V\in C([0,T];\mathcal A^{s+1})
\cap
C^1([0,T];\mathcal A^{s-1})
\]
on every finite horizon, together with a one-derivative smoothing estimate. This linear theory is subsequently used as the basic solution operator for the HJB iteration.
To obtain this result, we construct a frozen-symbol parametrix adapted to the augmented Barron structure. The main difficulty is to control the frozen Gaussian amplitude
\[
M_{t,r}(x;\xi)
=
\exp\left(
-\frac12\int_t^r\xi^\top a(\tau,x)\xi\,d\tau
\right)
\]
not only pointwise in \(x\), but in the augmented Barron norm. Using Wiener--Lévy inversion, GRS spectral invariance, and holomorphic functional calculus, we prove the uniform estimate
\[
\|M_{t,r}(\cdot;\xi)\|_{\mathcal A^{s+2}}
\le
C\exp\bigl(-c(r-t)|\xi|^2\bigr).
\]
The corresponding spatial derivative bounds imply that the parametrix defect satisfies
\[
\|R_{t,r}\|_{\mathcal B^s\to\mathcal B^s}
\le
C(r-t)^{-1/2}.
\]
Since this singularity is integrable, the defect can be inverted by a Volterra series on every finite horizon. The resulting Green and terminal operators provide the desired linear solution theory and smoothing estimate, without either a perturbative smallness condition on the spatial variation of the principal coefficient or a short-time restriction.

We next use this linear framework to treat the nonlinear HJB equation. After explicit minimization over the control, we employ a semi-explicit gradient iteration, following the gradient-iteration framework of Kerimkulov, Šiška, and Szpruch \cite{kerimkulov2020}. At each step, the full state-dependent second-order operator is retained, while the nonlinear Hamiltonian is evaluated at the preceding iterate. Consequently, every iteration step is a linear terminal-value problem with the same variable principal part and is therefore covered by the preceding Barron-space framework. The one-derivative smoothing estimate then propagates through the iteration and yields, on sufficiently short horizons, convergence in the whole-space augmented Barron norm. More precisely, the iterates and the associated feedbacks satisfy a Gamma-factorial estimate of the form

\[
\sup_{t\in[0,T]}
\|V^n(t)-V^*(t)\|_{\mathcal A^{s+1}}
+
\sup_{t\in[0,T]}
\|u^n(t)-u^*(t)\|_{\mathcal B^s}
\le
C\frac{\Lambda^n}{\Gamma(n/2+1)}.
\]

The limit is a bounded classical solution of the HJB equation. A verification argument further identifies it with the stochastic-control value function and establishes optimality of the associated feedback. The short-horizon condition enters only at this nonlinear stage, through the uniform control of the iterates in an augmented Barron ball.

The HJB analysis yields spectral Barron regularity in space and Banach-valued regularity in time, but not joint Fourier Barron regularity in \((t,x)\). We therefore approximate the two variables separately: a Banach-valued Jackson approximation gives a finite cosine expansion in time with spatial Barron coefficients, and each coefficient is then approximated by a shallow cosine network. Using the product-to-sum identity, these temporal and spatial approximations are combined into a single shallow network in \((t,x)\). This yields quantitative joint space--time approximations of the value function and, under an additional time-regularity assumption on \(g\), the optimal feedback, with the error decomposing into the HJB iteration, temporal Jackson, and spatial Barron approximation errors.

The main contributions of the paper are summarized as follows.

1. We develop a whole-space augmented spectral Barron framework for uniformly elliptic parabolic equations with time-dependent, state-dependent principal coefficients in nondivergence form. The resulting Green and terminal operators are valid on arbitrary finite horizons and satisfy a one-derivative smoothing estimate, without a perturbative smallness assumption on the spatial variation of the principal coefficient.

2. We apply the linear framework to a finite-horizon HJB equation and prove short-horizon convergence of a semi-explicit gradient iteration at a Gamma-factorial rate in whole-space augmented Barron norms. The limiting classical solution is identified with the stochastic-control value function, and the associated feedback is shown to be admissible and optimal.

3. We combine temporal Jackson approximation with spatial spectral Barron approximation to obtain quantitative shallow-network approximations of the value function and optimal feedback in the joint space--time variables, without imposing joint space--time spectral Barron regularity.

\section{Preliminaries and Main Results}
\subsection{Spectral and Augmented Barron Spaces}
\begin{definition}[Spectral and augmented Barron spaces]\label{def:barron}
Let $\mathcal{S}(\R^d)$ and $\mathcal{S}'(\R^d)$ denote the Schwartz space and tempered distributions, respectively, with vector- and matrix-valued spaces understood componentwise. For $s\in\R$, we follow \cite[Definition 1]{feng2026} and use the Fourier convention
\begin{equation}
\hat{h}(\xi)=\frac{1}{(2\pi)^d}\int_{\R^d}h(x)e^{-ix\cdot\xi}\,dx;\qquad h(x)=\int_{\R^d}\hat{h}(\xi)e^{ix\cdot\xi}\,d\xi,\label{eq:fourier-convention}
\end{equation}
first on Schwartz functions and then on $\mathcal{S}'$ by duality.
For $E=\C$, $\C^m$, or $\C^{m\times n}$, define
\begin{equation}
B^s(\R^d;E):=\Bigl\{h\in\mathcal{S}'(\R^d;E):\norm{h}_{B^s(\R^d;E)}:=\int_{\R^d}|\hat{h}(\xi)|\,(1+|\xi|)^s\,d\xi<\infty\Bigr\}.\label{eq:barron-space}
\end{equation}
Here $|\cdot|$ denotes the scalar modulus, the vector $\ell^1$-norm, or the entrywise matrix $\ell^1$-norm, and $|\xi|$ is Euclidean. For $s\ge 0$, define
\begin{equation}
A^s(\R^d;E):=\{c+h:c\in E,\ h\in B^s(\R^d;E)\},\label{eq:augmented-space}
\end{equation}
where
\begin{equation}
\norm{c+h}_{A^s}:=|c|+\norm{h}_{B^s}.\label{eq:augmented-norm}
\end{equation}
\end{definition}
\begin{remark}[Unique augmented decomposition]\label{rem:unique}
Here $|\cdot|$ in \eqref{eq:augmented-norm} is the finite-dimensional $\ell^1$ norm specified above. Since $B^s\subset B^0\subset C_0(\R^d;E)$ for $s\ge 0$, the space $B^s$ contains no nonzero constant functions. The decomposition is therefore unique: if $c+h=c'+h'$, then the constant function $c-c'=h'-h$ belongs to $B^s$ and hence must vanish. Therefore $c=c'$ and $h=h'$, so
\[
A^s(\R^d;E)=E\oplus B^s(\R^d;E)
\]
is an algebraic direct sum, and \eqref{eq:augmented-norm} and the constant component of an augmented Barron function are well defined.
\end{remark}
\subsection{Finite-Horizon Stochastic Control and the HJB Equation}
Fix once and for all a reference horizon $\bar T>0$. All coefficient data below are defined on $[0,\bar T]$, and every statement for a horizon $0<T\le\bar T$ is understood with these data restricted to $[0,T]$. Let
\[
(\Omega,\mathcal{F},(\mathcal{F}_r)_{0\le r\le T},\mathbb{P})
\]
be a complete filtered probability space satisfying the usual conditions, and let $W=(W_r)_{0\le r\le T}$ be a $d$-dimensional Brownian motion. Fix a symmetric positive-definite matrix $R=R^{\top}>0$.
\begin{definition}[Admissible controls]\label{def:admissible}
Fix $(t;x)\in[0,T]\times\R^d$. An $\R^m$-valued, $(\mathcal{F}_r)$-progressively measurable process $u=(u_r)_{r\in[t,T]}$ is admissible for $(t;x)$ if the state equation
\begin{equation}
\begin{cases}
dX_r^{t,x;u}=\bigl(f(r,X_r^{t,x;u})+g(r,X_r^{t,x;u})u_r\bigr)\,dr+\sigma(r,X_r^{t,x;u})\,dW_r,\\
X_t^{t,x;u}=x
\end{cases}\label{eq:controlled-sde}
\end{equation}
admits a unique strong solution and
\[
\E\Bigl[\sup_{t\le r\le T}|X_r^{t,x;u}|^2\Bigr]+\E\Bigl[\int_t^T|u_r|^2\,dr\Bigr]<\infty.
\]
In addition, the cost functional
\begin{equation}
J(t;x;u):=\E\Bigl[\int_t^Te^{-\gamma (r-t)}\bigl(\ell(r,X_r^{t,x;u})+u_r^{\top}Ru_r\bigr)\,dr+e^{-\gamma (T-t)}\varphi(X_T^{t,x;u})\Bigr]\label{eq:cost-functional}
\end{equation}
must be well defined as an extended real number and satisfy $J(t;x;u)>-\infty$. The corresponding class is denoted by $\mathcal{U}_{t,T}(x)$.
\end{definition}
The value function is
\begin{equation}
V_{\mathrm{val}}(t;x):=\inf_{u\in\mathcal{U}_{t,T}(x)}J(t;x;u).\label{eq:value-function}
\end{equation}
Writing
\(a(t;x):=\sigma(t;x)\sigma(t;x)^{\top}.\)

\begin{definition}[Control Hamiltonian]\label{def:hamiltonian}
The control Hamiltonian is
\begin{equation}
H(t;x;p;u):=\bigl(f(t;x)+g(t;x)u\bigr)\cdot p+\ell(t;x)+u^{\top}Ru.\label{eq:control-hamiltonian}
\end{equation}
Because $R$ is positive definite, the map $u\mapsto H(t;x;p;u)$ is strictly convex. Its first-order condition
\[
2Ru+g(t;x)^{\top}p=0
\]
therefore gives the unique pointwise minimizing selector
\begin{equation}
u^{*}(t;x;p):=-\frac12\,R^{-1}g(t;x)^{\top}p.\label{eq:optimal-control}
\end{equation}
At this stage $u^{*}$ denotes only the pointwise Hamiltonian minimizer; the admissibility and optimality of the induced closed-loop control are proved in Section 4.
\end{definition}
\begin{definition}[Reduced Hamiltonian]\label{def:reduced}
Set
\(
Q(t;x):=g(t;x)R^{-1}g(t;x)^{\top}
\)
and define the reduced Hamiltonian by
\begin{equation}
F(t;x;p;v):=\inf_{u\in\R^m}H(t;x;p;u)-\gamma v
=f(t;x)\cdot p+\ell(t;x)-\tfrac14p^\top Q(t;x)p-\gamma v.
\label{eq:reduced-hamiltonian}
\end{equation}
For a function $V$, we write $F(t;x;\nabla_xV(t;x),V(t;x))$.
\begin{equation}
\partial_t V(t;x)+\frac12\,a(t;x):D_x^2 V(t;x)+F(t;x;\nabla_xV(t;x),V(t;x))=0,\qquad V(T;x)=\varphi(x).\label{eq:reduced-hjb}
\end{equation}
\end{definition}
\subsubsection{Standing assumptions}
For a Banach space $X$, the spaces $C([0,T];X)$ and $C^1([0,T];X)$
have their usual meanings, and
$\|v\|_{C([0,T];X)}=\sup_{t\in[0,T]}\|v(t)\|_X$.
Derivatives at time endpoints are understood one-sidedly.

\begin{assumption}[Standing assumptions]\label{ass:standing}
Let $s>2$. We assume
\[
f\in C\bigl([0,\bar T];A^s(\R^d;\R^d)\bigr),\qquad g\in C\bigl([0,\bar T];A^s(\R^d;\R^{d\times m})\bigr),
\]
\[
\ell\in C\bigl([0,\bar T];A^s(\R^d)\bigr),\qquad \varphi\in A^{s+1}(\R^d),\qquad R=R^{\top}>0,\qquad \gamma\in\mathbb{R}.
\]
The diffusion coefficient satisfies
\(\sigma\in C\bigl([0,\bar T];A^{s+2}(\R^d;\R^{d\times d})\bigr).\)
The matrix $a=\sigma\sigma^\top$ is uniformly elliptic: there is $\lambda>0$ such that
\begin{equation}
\xi^{\top}a(t;x)\xi\ge\lambda|\xi|^2,\qquad (t;x;\xi)\in[0,\bar T]\times\R^d\times\R^d.\label{eq:ellipticity}
\end{equation}
\end{assumption}

\begin{remark}[Large spatial variation]
\label{rem}
Let \(h(x)=e^{-|x|^2}\), so that \(h\in B^r(\mathbb R^d)\) for every \(r\ge0\). For any \(\rho\ge0\), set
\( \sigma_\rho(x)=(1+\rho h(x))I_d. \)
Then \(\sigma_\rho\in A^{s+2}\) and
\[ a_\rho(x)=\sigma_\rho(x)\sigma_\rho(x)^\top =(1+\rho h(x))^2I_d\ge I_d. \]
Meanwhile,
\[ \|\sigma_\rho-I_d\|_{A^{s+2}}\to\infty \qquad\text{as }\rho\to\infty. \]
Thus the standing assumptions allow arbitrarily large spatial variations within the augmented Barron class while preserving uniform ellipticity.

\end{remark}

\subsection{Semi-explicit gradient iteration}
The iteration below keeps the principal operator fixed and evaluates the
gradient nonlinearity explicitly at the previous iterate.
\begin{algorithm}[H]
\caption{Semi-explicit gradient iteration}
\label{alg:gradient-iteration}

\textbf{Step 1.} Initialize
\[
V^0=0,\qquad u^0=0.
\]

\textbf{Step 2.} For $n\ge0$, compute $V^{n+1}$ as the solution of
\[
\begin{cases}
\displaystyle
\partial_tV^{n+1}
+\frac12a:D_x^2V^{n+1}
+F(t,x,\nabla_xV^n,V^n)=0,
\\[1mm]
V^{n+1}(T)=\varphi .
\end{cases}
\]

\textbf{Step 3.} Update the feedback control by
\[
u^{n+1}
=
-\frac12R^{-1}g^\top\nabla_xV^{n+1}.
\]

\end{algorithm}
The principal operator is fixed; the gradient nonlinearity is explicit.

\subsection{Main Results}
\begin{theorem}[Linear well-posedness on arbitrary finite horizons]
\label{thm:linear-wellposedness}
Let $s>2$, $a=a^\top\in C([0,\bar T];A^{s+2})$, and $a(t,x)\ge\lambda I$
for some $\lambda>0$. For every $0<T\le\bar T$,
$H\in C([0,T];A^s)$, and $\varphi\in A^{s+1}$, the problem
\begin{equation}
\partial_tV+\tfrac12a:D_x^2V+H=0,\qquad V(T)=\varphi
\label{eq:linear-main-pde}
\end{equation}
has a unique solution
$V\in C([0,T];A^{s+1})\cap C^1([0,T];A^{s-1})$.
The operators constructed in Section~3 give
\begin{equation}
V(t)=U_A(t;T)\varphi+(G_AH)(t),
\label{eq:augmented-representation}
\end{equation}
and
\begin{equation}
\|V(t)\|_{A^{s+1}}
\le C_{\mathrm{ter},A}\|\varphi\|_{A^{s+1}}
+K_{G,A}\int_t^T(r-t)^{-1/2}\|H(r)\|_{A^s}\,dr.
\label{eq:augmented-bound}
\end{equation}
For equal terminal data,
\begin{equation}
\|V_1(t)-V_2(t)\|_{A^{s+1}}
\le K_{G,A}\int_t^T(r-t)^{-1/2}\|H_1(r)-H_2(r)\|_{A^s}\,dr.
\label{eq:augmented-difference-bound}
\end{equation}
The constants\(K_{G,A}\) and \(C_{ter,A}\) are uniform for $0<T\le\bar T$; they may depend on the
fixed coefficient family, its ellipticity, and $\bar T$.
\end{theorem}
\begin{theorem}[Solvability and Gamma-factorial convergence in augmented spectral Barron spaces]\label{thm:solvability}
Fix the reference horizon $\bar T>0$ of Assumption~\ref{ass:standing}. Then there exists $T_0\in(0,\bar T]$, depending only on the data in Assumption~\ref{ass:standing} and the reference horizon $\bar T$, such that, for every $0<T\le T_0$, with the coefficient data of Assumption~\ref{ass:standing} restricted from $[0,\bar T]$ to $[0,T]$, the reduced HJB equation \eqref{eq:reduced-hjb} admits a unique classical solution
\[
V^{*}\in C^1\bigl([0,T];A^{s-1}(\R^d)\bigr)\cap C\bigl([0,T];A^{s+1}(\R^d)\bigr).
\]
The stochastic-control value function $V_{\mathrm{val}}$ is finite and $V_{\mathrm{val}}(t;x)=V^{*}(t;x)$.
Define the optimal feedback by
\begin{equation}
u^{*}(t;x):=-\frac12\,R^{-1}g(t;x)^{\top}\nabla_x V^{*}(t;x).\label{eq:optimal-feedback}
\end{equation}
The iterates of Algorithm~\ref{alg:gradient-iteration} converge in
$C([0,T];A^{s+1})$. Moreover, there exist constants
\[
C=C(f,g,\ell,\varphi,R,\sigma,\bar T,\gamma)>0,
\qquad
\Lambda=\Lambda(f,g,\ell,\varphi,R,\sigma,\bar T,\gamma)>0,
\]
independent of $n$, such that, for every $n\ge0$,
\begin{equation}
\sup_{t\in[0,T]}
\|V^n(t)-V^{*}(t)\|_{A^{s+1}}
+
\sup_{t\in[0,T]}
\|u^n(t)-u^{*}(t)\|_{B^s}
\le
C\,
\frac{\Lambda^n}{\Gamma(n/2+1)}.
\label{eq:main-convergence}
\end{equation}
where $\Gamma$ denotes the Euler Gamma function. Furthermore, $u^{*}$ is admissible and optimal.
\end{theorem}
A $Q$-neuron shallow cosine network on $[0,T]\times\R^d$ is a function of the form
\[
\bar v(t;x)=\sum_{q=1}^{Q}c_q\cos(\alpha_q t+\beta_q\cdot x+\delta_q),
\]
with $c_q\in\R$ for scalar output and $c_q\in\R^m$ for vector output, where $\alpha_q,\delta_q\in\R$ and $\beta_q\in\R^d$.

\begin{theorem}[Joint space--time shallow neural-network approximation of the value function and optimal feedback]\label{thm:approx}
Let Assumption~\ref{ass:standing} hold and let \(0<T\le T_0\). Let \(V^{*}\), \(u^{*}\), and \(\{(V^n,u^n)\}_{n\ge0}\) be as in Theorem~\ref{thm:solvability}.
For every compact \(K\subset\R^d\), every \(n\ge0\), and every pair of integers \(N_t,N_x\ge1\), there exists a scalar-valued shallow cosine network \(\bar V^{n,N_t,N_x}\) satisfying \eqref{eq:value-network}, with
\(
Q_V\le(2N_t+1)(N_x+1).
\)
\begin{equation}
\begin{aligned}
\bar V^{n,N_t,N_x}(t;x)={}&c^n_{0,0}+\sum_{j=1}^{N_x}c^n_{0,j}\cos\bigl(\beta^n_{0,j}\cdot x+\delta^n_{0,j}\bigr)+\sum_{k=1}^{N_t}c^n_{k,0}\cos\Bigl(\frac{k\pi t}{T}\Bigr)\\
&+\sum_{k=1}^{N_t}\sum_{j=1}^{N_x}\sum_{l\in\{-1,1\}}c^n_{k,j,l}\cos\Bigl(\frac{k\pi t}{T}+l\bigl(\beta^n_{k,j}\cdot x+\delta^n_{k,j}\bigr)\Bigr),
\end{aligned}\label{eq:value-network}
\end{equation}
Moreover, there exist constants \(C_V>0\) and \(\Lambda>0\), independent of \(K,n,N_t,N_x\), such that \eqref{eq:value-error} holds.
\begin{equation}
\begin{aligned}
\sup_{t\in[0,T]}
\norm{V^{*}(t;\cdot)-\bar V^{n,N_t,N_x}(t;\cdot)}_{L^2(K)}
\le{}&
|K|^{1/2} C_V
\Bigl[
\frac{\Lambda^n}{\Gamma(n/2+1)}
\\
&
+N_t^{-1}
+(1+\log(N_t+1))N_x^{-1/2}
\Bigr].
\label{eq:value-error}
\end{aligned}
\end{equation}
If, in addition,
\begin{equation}
g\in C^1\bigl([0,\bar T];A^s(\R^d;\R^{d\times m})\bigr),
\label{eq:g-regularity}
\end{equation}
then there also exists an \(\R^m\)-valued shallow cosine network
\(\bar u^{n,N_t,N_x}\) satisfying \eqref{eq:feedback-network}, with
\(
Q_u\le(2N_t+1)N_x,
\)
\begin{equation}
\begin{aligned}
\bar u^{n,N_t,N_x}(t;x)={}&\sum_{j=1}^{N_x}d^n_{0,j}\cos\bigl(\eta^n_{0,j}\cdot x+\theta^n_{0,j}\bigr)\\
&+\sum_{k=1}^{N_t}\sum_{j=1}^{N_x}\sum_{l\in\{-1,1\}}d^n_{k,j,l}\cos\Bigl(\frac{k\pi t}{T}+l\bigl(\eta^n_{k,j}\cdot x+\theta^n_{k,j}\bigr)\Bigr),
\end{aligned}\label{eq:feedback-network}
\end{equation}

and there exists \(C_u^{\mathrm{app}}>0\), independent of \(K,n,N_t,N_x\), such that \eqref{eq:feedback-error} holds.

\begin{equation}
\begin{aligned}
\sup_{t\in[0,T]}
\norm{u^{*}(t;\cdot)-\bar u^{n,N_t,N_x}(t;\cdot)}
_{L^2(K;\R^m)}
\le{}&
|K|^{1/2} C_u^{\mathrm{app}}
\Bigl[
\frac{\Lambda^n}{\Gamma(n/2+1)}
\\
&
+N_t^{-1}
+(1+\log(N_t+1))N_x^{-1/2}
\Bigr].
\label{eq:feedback-error}
\end{aligned}
\end{equation}
For $n\ge1$, the same value network also satisfies
\[
\|\varphi-\bar V^{n,N_t,N_x}(T)\|_{L^2(K)}
\le |K|^{1/2}C_{V}\bigl(N_t^{-1}+(1+\log(N_t+1))N_x^{-1/2}\bigr).
\]
\end{theorem}
Here $N_t$ is the maximal temporal cosine mode and $N_x$ is the spatial approximation budget assigned to each temporal mode. The three terms on the right-hand side are, respectively, the HJB iteration error, the temporal Jackson error, and the accumulated spatial Barron approximation error. The logarithmic factor arises from the $k^{-1}$ decay of the Banach-valued temporal cosine coefficients.

\section{Barron-Space Tools for the Variable-Coefficient Parabolic Problem}
The construction follows the classical Levi parametrix method and its parabolic developments; see Levi~\cite{levi1907} and, for variable-coefficient parabolic fundamental-solution and parametrix theory, \cite{aronson1967,eidelman1998,friedman1964} and \cite[Chapter 7, Section 13]{taylor2023}. The principal quadratic symbol is frozen at the spatial base point and its Gaussian exponential is corrected by a Volterra series. Unlike the smooth pseudodifferential setting, the estimates here are global in augmented Barron and weighted Fourier--$L^1$ norms. Wiener--L\'evy and GRS spectral invariance control the frozen multiplier, while the defect has the integrable order $(r-t)^{-1/2}$.
\subsection{Basic Barron-Space Estimates}
\begin{proposition}[Basic Barron estimates]\label{prop:basic}
Let $s\ge 0$. Then $B^{s+1}\hookrightarrow B^s\hookrightarrow B^0\hookrightarrow C_0(\R^d)$, and, for $i,j=1,\ldots,d$, $\norm{\partial_i h}_{B^s}\le\norm{h}_{B^{s+1}}$ and $\norm{\partial_{ij}h}_{B^{s-1}}\le\norm{h}_{B^{s+1}}$. For scalar functions,
\[
\norm{h_1h_2}_{B^s}\le\norm{h_1}_{B^s}\norm{h_2}_{B^s}.
\]
These estimates extend componentwise to finite-dimensional vector and matrix spaces.
\end{proposition}
\begin{proof}
The scale embeddings follow directly from the definition of the spectral Barron norm; see also \cite[Proposition~3.2]{feng2026}. The inclusion
\(
B^0(\R^d)\hookrightarrow C_0(\R^d)
\)
follows from Fourier inversion and the Riemann--Lebesgue lemma. The differentiation estimates follow directly from the Fourier definition; see also \cite[Proposition~3.3]{feng2026} for the first-order estimate. The second-order estimate follows by iteration. The product estimate follows from the convolution identity and the submultiplicativity of the weight.
\end{proof}
\begin{proposition}[Augmented Barron algebra estimates]\label{prop:algebra}
Let $s\ge 0$. For $a,a_1,a_2\in A^s$, $h\in B^s$, and $i,j=1,\ldots,d$,
\[
\norm{a_1a_2}_{A^s}\le\norm{a_1}_{A^s}\norm{a_2}_{A^s},\qquad \norm{a_1h}_{B^s}\le\norm{a_1}_{A^s}\norm{h}_{B^s}.
\]
For $a\in A^r$, the differentiation bounds are
\[
\|\partial_i a\|_{B^{r-1}}\le\|a\|_{A^r}\quad(r\ge1),\qquad
\|\partial_{ij}a\|_{B^{r-2}}\le\|a\|_{A^r}\quad(r\ge2).
\]
\end{proposition}
The proof is postponed to Appendix~A.1.

\subsection[Variable-Coefficient Parabolic Terminal Problem and Frozen Fourier Parametrix]{\texorpdfstring{Variable-Coefficient Parabolic Terminal Problem and Frozen Fourier Parametrix}{Variable-Coefficient Parabolic Terminal Problem and Frozen Fourier Parametrix}}
Throughout Section~3, assume only the hypotheses on $a$ in
Theorem~\ref{thm:linear-wellposedness}, and set
$\Lambda_a=\sup_t\|a(t)\|_{A^{s+2}}$.
We consider the terminal-value problem
\begin{equation}
\partial_t v(t;x)+\frac12\,a(t;x):D_x^2 v(t;x)+H(t;x)=0,\qquad v(T;x)=\varphi(x).\label{eq:linear-pde}
\end{equation}
Set
\[
(L_t v)(x):=\frac12\,a(t;x):D_x^2 v(x),\qquad (\A v)(t;x):=-\partial_t v(t;x)-\bigl(L_t[v(t;\cdot)]\bigr)(x).
\]
Thus \eqref{eq:linear-pde} is $(\A v)(t;x)=H(t;x)$ with $v(T;x)=\varphi(x)$. Our goal is to construct a zero-terminal Green operator
\[
G_A:C\bigl([0,T];A^s\bigr)\to C\bigl([0,T];A^{s+1}\bigr)
\]
and a terminal propagator $U_A(t;T):A^{s+1}\to A^{s+1}$ such that
\[
\bigl(\A(G_AH)\bigr)(t;x)=H(t;x),\qquad (G_AH)(T;x)=0,
\]
\[
\bigl(\A[U_A(\cdot;T)\varphi]\bigr)(t;x)=0,\qquad \bigl(U_A(T;T)\varphi\bigr)(x)=\varphi(x).
\]
The exact solution will then be
\begin{equation}
v(t;x)=\bigl(U_A(t;T)\varphi\bigr)(x)+(G_AH)(t;x).\label{eq:representation}
\end{equation}
\begin{proposition}[Diffusion coefficient]\label{prop:diffusion}
Under Assumption~\ref{ass:standing},
\begin{equation}
a=\sigma\sigma^{\top}\in C\bigl([0,\bar T];A^{s+2}(\R^d;\R^{d\times d})\bigr).\label{eq:diffusion-regularity}
\end{equation}
With
\(
\Lambda_a:=\sup_{t\in[0,\bar T]}\norm{a(t)}_{A^{s+2}}<\infty,
\)
one has
\[
\lambda I\le a(t;x)\le\Lambda_a I
\]
in the sense of quadratic forms.
\end{proposition}
\begin{proof}
See Appendix A.1.
\end{proof}
\subsection{Frozen Principal Symbol and the Wiener--GRS Gaussian Estimate}
For the Fourier-side description, set
\(
w_s(\xi):=(1+|\xi|)^s
\)
and define the weighted Wiener algebra
\[
L^1_{w_s}(\R^d):=\Bigl\{u:\R^d\to\C\text{ measurable}:\norm{u}_{L^1_{w_s}}:=\int_{\R^d}|u(\xi)|\,w_s(\xi)\,d\xi<\infty\Bigr\}.
\]
With convolution as its product, the Fourier transform gives the isometric algebra correspondence
\[
\mathcal{F}:B^s(\R^d)\to L^1_{w_s}(\R^d),\qquad h\mapsto\hat h.
\]
Let $\delta_0$ denote the Dirac mass at the origin and define the unitized weighted Wiener algebra
\[
\widetilde L{}^1_{w_s}:=\C\delta_0\oplus L^1_{w_s},\qquad \norm{c\delta_0+u}_{\widetilde L{}^1_{w_s}}:=|c|+\norm{u}_{L^1_{w_s}}.
\]
Since
\(
\mathcal{F}(c+h)=c\delta_0+\hat h,
\)
the augmented space has the corresponding isometric unital algebra identification
\[
\mathcal{F}:A^s\to\widetilde L{}^1_{w_s},\qquad c+h\mapsto c\delta_0+\hat h.
\]
Thus adjoining spatial constants in $A^s$ is exactly the unitization of the Fourier-side Wiener algebra. The Wiener--L\'evy and GRS lemmas used below are proved in Appendix A. Vector- and matrix-valued spaces are understood componentwise.
\begin{remark}[Complexification]\label{rem:complex}
Fourier analysis and holomorphic functional calculus use the complexified
spaces. All PDE and control data below are real-valued, as are the
solutions; the linear result also extends to complex data by taking
real and imaginary parts.
\end{remark}
The principal symbol is
\(
p(t;x;\xi):=\frac12\,\xi^{\top}a(t;x)\xi.
\)
If $a$ were independent of $x$, its time-integrated exponential would be an ordinary Fourier multiplier. Here it depends on $x$ and becomes the amplitude of the frozen parametrix constructed below.
We fix the normalization used throughout Subsections 3.3--3.4. For $0\le t<r\le\bar T$ and $\xi\neq 0$, write
\begin{equation}
\omega:=\frac{\xi}{|\xi|},\qquad q:=(r-t)|\xi|^2,\qquad \vartheta_{t,r,\omega}(x):=\frac{1}{2(r-t)}\int_t^r\omega^{\top}a(\tau;x)\omega\,d\tau,\label{eq:theta}
\end{equation}
and extend the last quantity to the diagonal by $\vartheta_{t,t,\omega}(x):=\frac12\omega^{\top}a(t;x)\omega$ for $\omega\in S^{d-1}$.
\begin{lemma}[Normalized frozen-coefficient family]\label{lem:family}
Let $\vartheta_{t,r,\omega}$ be defined by \eqref{eq:theta} and its diagonal extension. Then
\[
\mathcal{K}:=\{\vartheta_{t,r,\omega}:0\le t\le r\le\bar T,\ \omega\in S^{d-1}\}
\]
is compact in $A^{s+2}$.
\end{lemma}
\begin{proof}
Because $a\in C([0,\bar T];A^{s+2})$, its time averages and their finite contractions against $\omega$ belong to $A^{s+2}$. Thus it is enough to verify continuity of the parameter map. Continuity away from $r=t$ is immediate. At the diagonal,
\[
\Bigl\lVert\frac{1}{r-t}\int_t^r a(\tau;x)\,d\tau-a(t;x)\Bigr\rVert_{A^{s+2}}\le\sup_{\tau\in[t,r]}\norm{a(\tau;x)-a(t;x)}_{A^{s+2}}\to 0,
\]
and continuity in $\omega$ follows from the finite componentwise sum. Thus $(t;r;\omega)\mapsto\vartheta_{t,r,\omega}$ is continuous on the compact set $\{0\le t\le r\le\bar T\}\times S^{d-1}$, whose image is $\mathcal{K}$.
\end{proof}
\begin{proposition}[Uniform frozen-multiplier estimate]\label{prop:frozen}
For $0\le t<r\le\bar T$ and $\xi\in\R^d$, define
\begin{equation}
M_{t,r}(x;\xi):=\exp\Bigl(-\frac12\int_t^r\xi^{\top}a(\tau;x)\xi\,d\tau\Bigr).\label{eq:gaussian-multiplier}
\end{equation}
There exists $C_M>0$, depending only on the coefficient family and the ellipticity bounds, such that
\begin{equation}
\norm{M_{t,r}(\cdot;\xi)}_{A^{s+2}}\le C_M\exp\Bigl(-\frac{\lambda}{4}(r-t)|\xi|^2\Bigr).\label{eq:gaussian-decay}
\end{equation}
\end{proposition}
\begin{proof}
Set $\rho:=s+2$. For $\vartheta\in\mathcal{K}$, uniform ellipticity gives
\begin{equation}
\frac{\lambda}{2}\le\vartheta(x)\le\frac{\Lambda_a}{2}.\label{eq:theta-bounds}
\end{equation}
For $\xi\neq 0$, \eqref{eq:theta} yields
\begin{equation}
M_{t,r}(x;\xi)=e^{-q\vartheta_{t,r,\omega}(x)},\qquad \omega=\frac{\xi}{|\xi|},\qquad q=(r-t)|\xi|^2.\label{eq:gaussian-form}
\end{equation}
It therefore suffices to prove, uniformly for $\vartheta\in\mathcal{K}$ and $q\ge 0$,
\begin{equation}
\norm{e^{-q\vartheta}}_{A^{\rho}}\le C_Me^{-(\lambda/4)q}.\label{eq:gaussian-regularity}
\end{equation}
We first establish a uniform resolvent bound. For $z\notin[\lambda/2,\Lambda_a/2]$, put $\psi_z:=z1-\vartheta\in A^{\rho}$. By \eqref{eq:theta-bounds}, $\inf_{x\in\R^d}|\psi_z(x)|>0$, so Lemma~\ref{lem:wienerlevy} gives $\psi_z^{-1}\in A^0$. Write $\psi_z=c+h$ and $u:=\hat h\in L^1_{w_\rho}\subset L^1$. The unitized Fourier identification above, with $s$ replaced by $\rho$, the spectral-shift identity, and Lemma~\ref{lem:grs} give
\begin{equation}
\psi_z^{-1}\in A^0\iff -c\notin\operatorname{spec}_{L^1}(u)\iff -c\notin\operatorname{spec}_{L^1_{w_\rho}}(u)\iff \psi_z^{-1}\in A^{\rho}.\label{eq:spectral-invariance}
\end{equation}
All spectra in \eqref{eq:spectral-invariance} are computed in the corresponding unitizations. Thus
\[
\operatorname{spec}_{A^{\rho}}(\vartheta)\subset[\lambda/2,\Lambda_a/2]\qquad(\vartheta\in\mathcal{K}).
\]
Let $\mathcal{C}$ be the positively oriented boundary of
\[
\Bigl\{z\in\C:\frac{\lambda}{4}\le\operatorname{Re}z\le\frac{\Lambda_a}{2}+\frac{\lambda}{4},\ |\operatorname{Im}z|\le\frac{\lambda}{4}\Bigr\}.
\]
It surrounds the displayed interval and satisfies $\operatorname{Re}z\ge\lambda/4$ and $\mathrm{length}(\mathcal{C})=\Lambda_a+\lambda$. Since $\mathcal{K}$ is compact by Lemma~\ref{lem:family}, continuity of inversion on the invertible group gives
\begin{equation}
C_{\mathrm{res}}:=\sup_{\vartheta\in\mathcal{K},\,z\in\mathcal{C}}\norm{(z1-\vartheta)^{-1}}_{A^{\rho}}<\infty.\label{eq:resolvent-bound}
\end{equation}
The holomorphic functional calculus \cite[Definition 10.26 and Theorem 10.27]{rudin1991} now gives
\[
e^{-q\vartheta}=\frac{1}{2\pi i}\int_{\mathcal{C}}e^{-qz}(z1-\vartheta)^{-1}\,dz,
\]
and
\[
\norm{e^{-q\vartheta}}_{A^{\rho}}\le\frac 1 {2\pi}|\mathcal{C}|e^{-\frac{\lambda q}{4}}C_{\mathrm{res}},
\]
Thus \eqref{eq:gaussian-regularity} holds with
\begin{equation}
C_M:=\max\Bigl\{1,\,C_{\mathrm{res}}\frac{\Lambda_a+\lambda}{2\pi}\Bigr\}.\label{eq:cm-def}
\end{equation}
Combining this with \eqref{eq:gaussian-form} proves \eqref{eq:gaussian-decay} for $\xi\neq 0$. For $\xi=0$, $M_{t,r}(x;0)=1$ for every $x$, which is covered by the definition of $C_M$.
\end{proof}
\begin{proposition}[Frozen-multiplier derivative bounds]\label{prop:deriv}
Let $0\le t<r\le T\le\bar T$. Set
\[
C_{\nabla M}:=C_M\Lambda_a;\qquad C_{D^2M}:=C_M\Bigl(\frac{\Lambda_a}{2}+\frac{\Lambda_a^2}{4}\Bigr).
\]
For every $\xi\in\R^d$, $x\in\R^d$ and $i,j=1,\ldots,d$,
\begin{equation}
\max_i\norm{\partial_i M_{t,r}(\cdot;\xi)}_{B^{s+1}}\le C_{\nabla M}\,q\exp\Bigl(-\frac{\lambda}{4}q\Bigr),\label{eq:first-derivative-bound}
\end{equation}
\begin{equation}
\max_{i,j}\norm{\partial_{ij}M_{t,r}(\cdot;\xi)}_{B^s}\le C_{D^2M}(q+q^2)\exp\Bigl(-\frac{\lambda}{4}q\Bigr).\label{eq:second-derivative-bound}
\end{equation}
\end{proposition}
\begin{proof}
For $\xi\neq 0$, \eqref{eq:gaussian-form} gives
\(M_{t,r}(x;\xi)=e^{-q\vartheta_{t,r,\omega}(x)}.\)
\begin{equation}
\partial_i M_{t,r}(x;\xi)=-qM_{t,r}(x;\xi)\partial_i\vartheta_{t,r,\omega}(x),\label{eq:d1-multiplier}
\end{equation}
\begin{equation}
\partial_{ij}M_{t,r}(x;\xi)=M_{t,r}(x;\xi)\bigl(q^2\partial_i\vartheta_{t,r,\omega}(x)\partial_j\vartheta_{t,r,\omega}(x)-q\partial_{ij}\vartheta_{t,r,\omega}(x)\bigr).\label{eq:d2-multiplier}
\end{equation}
The derivative bounds for $\vartheta_{t,r,\omega}$ follow from the
definition of $\vartheta$ and the differentiation estimates in
Proposition~\ref{prop:algebra}.
Moreover,
\[
\max_{i,j}\norm{\partial_{ij}\vartheta_{t,r,\omega}}_{B^s}
\le\max_i\norm{\partial_i\vartheta_{t,r,\omega}}_{B^{s+1}}
\le\norm{\vartheta_{t,r,\omega}}_{A^{s+2}}
\le\frac{\Lambda_a}{2}.
\]
By Proposition~\ref{prop:frozen} and the module estimate,
\[
\begin{aligned}
\norm{\partial_i M_{t,r}(\cdot;\xi)}_{B^{s+1}}&\le q\norm{M_{t,r}(\cdot;\xi)}_{A^{s+1}}\norm{\partial_i\vartheta_{t,r,\omega}}_{B^{s+1}}\le C_{\nabla M}\,qe^{-(\lambda/4)q};\\
\norm{\partial_{ij}M_{t,r}(\cdot;\xi)}_{B^s}&\le\norm{M_{t,r}(\cdot;\xi)}_{A^s}\Bigl(q^2\norm{\partial_i\vartheta_{t,r,\omega}}_{B^s}\norm{\partial_j\vartheta_{t,r,\omega}}_{B^s}+q\norm{\partial_{ij}\vartheta_{t,r,\omega}}_{B^s}\Bigr)\\
&\le C_Me^{-(\lambda/4)q}\Bigl(\frac{\Lambda_a^2}{4}q^2+\frac{\Lambda_a}{2}q\Bigr)\le C_{D^2M}(q+q^2)e^{-(\lambda/4)q}.
\end{aligned}
\]
For $\xi=0$, $M_{t,r}(x;0)=1$, so both derivatives vanish.
\end{proof}

\subsection{Frozen Parametrix and the Defect Operator}
We distinguish the two-parameter spatial family from the operators acting on time-dependent sources. For fixed $0\le t<r\le T$, define the frozen spatial operator
\begin{equation}
(P_{t,r}h)(x):=\int_{\R^d}e^{ix\cdot\xi}M_{t,r}(x;\xi)\hat h(\xi)\,d\xi,\qquad h\in B^s(\R^d).\label{eq:frozen-operator}
\end{equation}
For a time-dependent source $H\in C([0,T];B^s(\R^d))$, write
\(
\hat H(r;\xi):=\mathcal{F}_x[H(r;\cdot)](\xi).
\)
The source parametrix is the time-space operator
\begin{equation}
(\mathcal{P}H)(t;x):=\int_t^T\bigl(P_{t,r}[H(r;\cdot)]\bigr)(x)\,dr
=\int_t^T\!\int_{\R^d}e^{ix\cdot\xi}M_{t,r}(x;\xi)\hat H(r;\xi)\,d\xi\,dr.\label{eq:parametrix}
\end{equation}
Define the defect symbol
\begin{equation}
R_{t,r}(x;\xi):=-i\bigl(a(t;x)\xi\bigr)\cdot\nabla_x M_{t,r}(x;\xi)-\frac12\,a(t;x):D_x^2 M_{t,r}(x;\xi),\label{eq:defect-symbol}
\end{equation}
and the associated two-parameter spatial operator
\begin{equation}
(\mathcal{R}_{t,r}h)(x):=\int_{\R^d}e^{ix\cdot\xi}R_{t,r}(x;\xi)\hat h(\xi)\,d\xi.\label{eq:defect-operator}
\end{equation}
The Volterra defect operator acting on the time-dependent source is
\begin{equation}
(\Sigma H)(t;x):=\int_t^T\bigl(\mathcal{R}_{t,r}[H(r;\cdot)]\bigr)(x)\,dr
=\int_t^T\!\int_{\R^d}e^{ix\cdot\xi}R_{t,r}(x;\xi)\hat H(r;\xi)\,d\xi\,dr.\label{eq:sigma-operator}
\end{equation}
Thus $P_{t,r}$ and $\mathcal{R}_{t,r}$ map spatial functions to spatial functions at fixed $(t;r)$, whereas $\mathcal{P}$ and $\Sigma$ map a time-dependent source $H(r;x)$ to functions of $(t;x)$.

We shall repeatedly use the following elementary symbol estimate. It is the step that converts an $A^s$ bound in the physical variable of an amplitude into a $B^s$ operator bound.
\begin{lemma}[Fourier-symbol estimate]\label{lem:symbol}
Let $s\ge0$, $h\in B^0(\R^d)$, and let $\xi\mapsto m(\cdot;\xi)\in A^s(\R^d)$ be strongly measurable. Assume that
\(
\int_{\R^d}\norm{m(\cdot;\xi)}_{A^s}\,|\hat h(\xi)|\,(1+|\xi|)^s\,d\xi<\infty.
\)
Define
\(
(T_mh)(x):=\int_{\R^d}e^{ix\cdot\xi}m(x;\xi)\hat h(\xi)\,d\xi
\)
then
\begin{equation}
\norm{T_mh}_{B^s}\le\int_{\R^d}\norm{m(\cdot;\xi)}_{A^s}\,|\hat h(\xi)|\,(1+|\xi|)^s\,d\xi.\label{eq:symbol-estimate}
\end{equation}
\end{lemma}
\begin{proof}
Write $m(\cdot;\xi)=c(\xi)+\tilde m(\cdot;\xi)$ with $\tilde m(\cdot;\xi)\in B^s$. The constant part is an ordinary Fourier multiplier. For the Barron part, the hypothesis makes the $B^s$-valued integral absolutely Bochner integrable. Since the Fourier transform is an isometry from $B^s$ to $L^1_{w_s}$, it commutes with this Bochner integral, and hence
\[
\widehat{T_{\tilde m}h}(\eta)=\int_{\R^d}\widehat{\tilde m}(\cdot;\xi)(\eta-\xi)\hat h(\xi)\,d\xi.
\]
The inequality $(1+|\eta|)^s\le(1+|\eta-\xi|)^s(1+|\xi|)^s$ and Tonelli's theorem give
\[
\norm{T_{\tilde m}h}_{B^s}\le\int_{\R^d}\norm{\tilde m(\cdot;\xi)}_{B^s}\,|\hat h(\xi)|\,(1+|\xi|)^s\,d\xi.
\]
The use of Tonelli is legitimate because the resulting integrand in $(\eta;\xi)$ is nonnegative; the displayed hypothesis makes its iterated integral finite. Adding the constant part proves \eqref{eq:symbol-estimate}.
\end{proof}
To identify its error, expand
\[
D_x^2\bigl(e^{ix\cdot\xi}M_{t,r}\bigr)=e^{ix\cdot\xi}\bigl(-\xi\xi^{\top}M_{t,r}+i\xi\otimes\nabla_x M_{t,r}+i\nabla_x M_{t,r}\otimes\xi+D_x^2 M_{t,r}\bigr).
\]
The quadratic term is cancelled exactly by $-\partial_t M_{t,r}(x;\xi)$. Consequently,
\begin{equation}
-\partial_t\bigl(e^{ix\cdot\xi}M_{t,r}(x;\xi)\bigr)-\tfrac12\,a(t;x):D_x^2\bigl(e^{ix\cdot\xi}M_{t,r}(x;\xi)\bigr)=e^{ix\cdot\xi}R_{t,r}(x;\xi),\label{eq:adjoint-identity}
\end{equation}
where $R_{t,r}$ is the defect symbol defined above.
The following proposition collects all pointwise-in-time estimates for the frozen parametrix and its defect. In particular, it supplies the weakly singular kernel used in the subsequent Volterra inversion.
\begin{proposition}[Parametrix bounds]\label{prop:parametrix}
There exist $C_P,C_\Sigma>0$, independent of $t,r,T$, such that for $0\le t<r\le T\le\bar T$, $h\in B^s$, and $\varphi\in B^{s+1}$, with
\[
C_{\mathrm{def}}:=\Lambda_a\Bigl(C_{\nabla M}+\frac12 C_{D^2M}\Bigr),
\]
\[
C_P:=C_M\max\Bigl\{1,\ \bar T^{1/2}+\Bigl(\frac{2}{e\lambda}\Bigr)^{1/2}\Bigr\},
\qquad
C_\Sigma:=C_{\mathrm{def}}\Bigl[\Bigl(\frac{6}{e\lambda}\Bigr)^{3/2}+\sqrt{\bar T}\Bigl(\frac{4}{e\lambda}+\frac{64}{e^2\lambda^2}\Bigr)\Bigr].
\]
Then
\begin{equation}
\norm{P_{t,r}h}_{B^s}\le C_P\norm{h}_{B^s},\label{eq:pt-bound}
\end{equation}
\begin{equation}
    \norm{P_{t,r}h}_{B^{s+1}}\le C_P(r-t)^{-1/2}\norm{h}_{B^s}\label{eq:pt-smoothing}
\end{equation}
and
\begin{equation}
\norm{P_{t,r}\varphi}_{B^{s+1}}\le C_P\norm{\varphi}_{B^{s+1}}.\label{eq:pt-terminal}
\end{equation}
Moreover, for every $h\in B^s$,
\begin{equation}
\norm{\mathcal{R}_{t,r}h}_{B^s}\le C_\Sigma(r-t)^{-1/2}\norm{h}_{B^s}.\label{eq:defect-bound}
\end{equation}
\end{proposition}
\begin{proof}
The defect bound \eqref{eq:defect-bound} is the most delicate estimate, so we prove it. The componentwise algebra and module estimates, the uniform bound $\norm{a(t)}_{A^{s+2}}\le\Lambda_a$, and \eqref{eq:first-derivative-bound}--\eqref{eq:second-derivative-bound} give
\begin{equation}
\begin{aligned}
\norm{R_{t,r}(\cdot;\xi)}_{A^s}&\le\Lambda_a\Bigl(|\xi|\max_i\norm{\partial_i M_{t,r}(\cdot;\xi)}_{B^{s+1}}+\frac12\max_{i,j}\norm{\partial_{ij}M_{t,r}(\cdot;\xi)}_{B^s}\Bigr)\\
&\le C_{\mathrm{def}}\bigl(|\xi|q+q+q^2\bigr)e^{-(\lambda/4)q}.
\end{aligned}\label{eq:defect-symbol-bound}
\end{equation}
Since
\[
|\xi|qe^{-(\lambda/4)q}\le\Bigl(\frac{6}{e\lambda}\Bigr)^{3/2}(r-t)^{-1/2},\qquad
(q+q^2)e^{-(\lambda/4)q}\le\frac{4}{e\lambda}+\frac{64}{e^2\lambda^2},
\]
and $1\le\bar T^{1/2}(r-t)^{-1/2}$, it follows that
\[
\norm{R_{t,r}(\cdot;\xi)}_{A^s}\le C_\Sigma(r-t)^{-1/2}
\]
uniformly in $\xi$. A final application of Lemma~\ref{lem:symbol} yields
\eqref{eq:defect-bound}. The remaining estimates for
$P_{t,r}$ are obtained by the same Fourier-symbol argument and are
provided in Appendix~A for completeness.
\end{proof}
Consequently, for every $H\in C([0,T];B^s)$,
\begin{equation}
\norm{(\mathcal P H)(t)}_{B^{s+1}}
\le
C_P\int_t^T(r-t)^{-1/2}\norm{H(r)}_{B^s}\,dr,
\qquad 0\le t\le T,
\label{eq:parametrix-bochner}
\end{equation}
and
\begin{equation}
\norm{(\Sigma H)(t)}_{B^s}
\le
C_\Sigma\int_t^T(r-t)^{-1/2}\norm{H(r)}_{B^s}\,dr,
\qquad 0\le t\le T.
\label{eq:sigma-bochner}
\end{equation}
\begin{lemma}[Continuity at the diagonal]\label{lem:diagonal}
Let $s>2$ be the exponent in Theorem~\ref{thm:linear-wellposedness}. Then, for every
$0\le r\le s+2$ and every $h\in B^r$,
\begin{equation}
\lim_{\varepsilon\to0}
\sup_{0\le t\le T-\varepsilon}
\norm{(P_{t,t+\varepsilon}-I)h}_{B^r}
=0.
\label{eq:diagonal-continuity}
\end{equation}
\end{lemma}

\begin{proof}
See Appendix A.1.
\end{proof}

\begin{proposition}[Frozen parametrix and defect identity]
\label{prop:identity}
For every $H\in C([0,T];B^s)$,
\[
\Sigma H\in C([0,T];B^s),
\]
and
\begin{equation}
\bigl(\A(\mathcal P H)\bigr)(t;x)
=
H(t;x)+(\Sigma H)(t;x),
\qquad
(\mathcal P H)(T;x)=0.
\label{eq:green-identity}
\end{equation}
\end{proposition}
\begin{proof}
Estimate \eqref{eq:sigma-bochner} follows immediately from
\eqref{eq:defect-bound}.
We give the continuity argument explicitly because it is also used
below when the truncated parametrix is differentiated.  For every
$\delta>0$, the maps
\[
(t,r)\longmapsto M_{t,r}(\cdot;\xi),\qquad
(t,r)\longmapsto R_{t,r}(\cdot;\xi)
\]
are continuous in $A^{s+2}$ and $A^s$, respectively, on
$\{(t,r):0\le t\le r-\delta\le T-\delta\}$.  Indeed, this follows from
the continuity of $a$ in time, continuity of the Bochner averages in
\eqref{eq:theta}, and continuity of multiplication and the exponential
map in the Banach algebra $A^{s+2}$.  Lemma~\ref{lem:symbol} and
dominated convergence give strong continuity of the fixed-input operators
away from the diagonal. For varying inputs, use
$T_nh_n-Th=T_n(h_n-h)+(T_n-T)h$ and the uniform operator bounds.
Time integration then gives continuity of the part of $\Sigma H$ with
$r-t\ge\delta$.  The complementary part is bounded,
uniformly in $t$, by
\[
C_\Sigma\norm{H}_{C([0,T];B^s)}
\int_t^{\min\{T,t+\delta\}}(r-t)^{-1/2}\,dr
\le 2C_\Sigma\sqrt\delta\,
\norm{H}_{C([0,T];B^s)}.
\]
Letting $\delta\downarrow0$ proves $\Sigma H\in C([0,T];B^s)$.
The same split, using \eqref{eq:pt-smoothing}, proves
$\mathcal P H\in C([0,T];B^{s+1})$: away from the diagonal, the
Fourier-symbol estimate and Gaussian decay give continuity in $B^{s+1}$;
the remaining integral has norm at most
$2C_P\sqrt\delta\,\norm{H}_{C([0,T];B^s)}$, uniformly in $t$.
For $\varepsilon>0$ and $0\le t\le T-\varepsilon$, define
\[
(\mathcal P_\varepsilon H)(t)
:=
\int_{t+\varepsilon}^T P_{t,r}[H(r)]\,dr,
\qquad
(\Sigma_\varepsilon H)(t)
:=
\int_{t+\varepsilon}^T
\mathcal R_{t,r}[H(r)]\,dr.
\]
Since $r-t\ge\varepsilon$, the preceding off-diagonal
continuity and the Gaussian symbol estimates provide an integrable
$A^s$-majorant for the time derivative and the required spatial
derivatives.  Hence differentiation under the Bochner integral is
legitimate. Using
\eqref{eq:adjoint-identity} and the Leibniz rule gives
\begin{equation}
\A(\mathcal P_\varepsilon H)(t)
=
P_{t,t+\varepsilon}[H(t+\varepsilon)]
+
(\Sigma_\varepsilon H)(t).
\label{eq:epsilon-identity}
\end{equation}
By \eqref{eq:pt-smoothing},
\[
\begin{aligned}
\norm{(\mathcal P-\mathcal P_\varepsilon)H(t)}_{B^{s+1}}
&\le
C_P\int_t^{t+\varepsilon}
(r-t)^{-1/2}\norm{H(r)}_{B^s}\,dr \\
&\le
2C_P\sqrt{\varepsilon}\,
\norm{H}_{C([0,T];B^s)},
\end{aligned}
\]
and similarly, by \eqref{eq:defect-bound},
\[
\norm{(\Sigma-\Sigma_\varepsilon)H(t)}_{B^s}
\le
2C_\Sigma\sqrt{\varepsilon}\,
\norm{H}_{C([0,T];B^s)}.
\]
Hence
\[
\mathcal P_\varepsilon H\to\mathcal P H
\quad\text{in }B^{s+1},
\qquad
\Sigma_\varepsilon H\to\Sigma H
\quad\text{in }B^s,
\]
uniformly for $0\le t\le T-\varepsilon$.
For the boundary term,
\[
\norm{
P_{t,t+\varepsilon}H(t+\varepsilon)-H(t)
}_{B^s}
\le
C_P\norm{H(t+\varepsilon)-H(t)}_{B^s} +\norm{(P_{t,t+\varepsilon}-I)H(t)}_{B^s}.
\]
The first term tends to zero by the uniform continuity of $H$.
The second tends to zero by Lemma~\ref{lem:diagonal}; the convergence
is uniform for $H(t)$ because $H([0,T])$ is compact in $B^s$ and
\eqref{eq:pt-bound} gives a uniform bound for $P_{t,t+\varepsilon}$.
Thus
\(
P_{t,t+\varepsilon}H(t+\varepsilon)\to H(t)
\text{ in }B^s.
\)
Passing to the limit in \eqref{eq:epsilon-identity} in the sense of
distributions yields
\[
\A(\mathcal P H)=H+\Sigma H.
\]
Finally, \eqref{eq:parametrix-bochner} gives
\[
\norm{(\mathcal P H)(t;x)}_{B^{s+1}}
\le
2C_P\sqrt{T-t}\,
\norm{H}_{C([0,T];B^s)}
\to0
\]
as $t\uparrow T$, and hence $(\mathcal P H)(T;x)=0$.
\end{proof}
\subsection{Volterra Inversion and the Green Operator}
For locally integrable kernels on $(0,\infty)$, define
\(
(f*g)(\tau):=\int_0^\tau f(\tau-r)g(r)\,dr.
\)
Let
\(
k(\tau):=\tau^{-1/2},
\kappa_1:=k,
\kappa_n:=k*\kappa_{n-1}, n\ge2.
\)
A standard Beta-function computation gives
\begin{equation}
\kappa_n(\tau)
=
\frac{\pi^{n/2}}{\Gamma(n/2)}
\tau^{n/2-1},
\qquad \tau>0.
\label{eq:volterra-kappa}
\end{equation}
In view of \eqref{eq:sigma-bochner}, repeated application of $\Sigma$
is therefore controlled by the kernels $C_\Sigma^n\kappa_n$.

\begin{proposition}[Zero-terminal Green operator]\label{prop:green}
\begin{equation}
(I+\Sigma)^{-1}=\sum_{n=0}^{\infty}(-\Sigma)^n\label{eq:resolvent-series}
\end{equation}
with absolute convergence in operator norm. The operator
\(
G:=\mathcal{P}(I+\Sigma)^{-1}
\)
satisfies, for every $H\in C([0,T];B^s)$,
\[
\bigl(\A(GH)\bigr)(t;x)=H(t;x),\qquad (GH)(T;x)=0,
\]
and preserves the pointwise half-order estimate
\begin{equation}
\norm{GH(t)}_{B^{s+1}}\le K_G\int_t^T(r-t)^{-1/2}\norm{H(r)}_{B^s}\,dr\label{eq:green-bound}
\end{equation}
where
\[
K_G:=\sum_{n=0}^{\infty}C_PC_\Sigma^n\,\frac{\pi^{(n+1)/2}\bar T^{n/2}}{\Gamma((n+1)/2)}.
\]
\end{proposition}
\begin{proof}
Iterating \eqref{eq:sigma-bochner}, applying Tonelli's theorem to the nonnegative norm majorant on the ordered time simplex, and using \eqref{eq:volterra-kappa} gives, for $n\ge 1$,
\begin{equation}
\norm{\Sigma^nH(t)}_{B^s}\le\frac{C_\Sigma^n\pi^{n/2}}{\Gamma(n/2)}\int_t^T(r-t)^{n/2-1}\norm{H(r)}_{B^s}\,dr.\label{eq:sigma-power}
\end{equation}
From \eqref{eq:sigma-power},
\[
\begin{aligned}
\sup_{t\in[0,T]}
\norm{\Sigma^nH(t)}_{B^s}
&\le
\frac{C_\Sigma^n\pi^{n/2}}{\Gamma(n/2)}
\norm{H}_{C([0,T];B^s)}
\sup_{t\in[0,T]}
\int_t^T(r-t)^{n/2-1}\,dr\\
&=
\frac{C_\Sigma^n\pi^{n/2}T^{n/2}}
{\Gamma(n/2+1)}
\norm{H}_{C([0,T];B^s)}.
\end{aligned}
\]
Therefore
\begin{equation}
\norm{\Sigma^n}_{C([0,T];B^s)\to C([0,T];B^s)}
\le
\frac{(C_\Sigma\sqrt{\pi T})^n}
{\Gamma(n/2+1)}.
\label{eq:sigma-norm}
\end{equation}
The Gamma denominator implies absolute convergence of the Neumann--Volterra series for every finite $T$ and every $C_\Sigma$; in particular, no condition such as $C_\Sigma T<1$ is required. To verify that the sum is the inverse, let
\(
S_N:=\sum_{n=0}^{N}(-\Sigma)^n, (I+\Sigma)S_N=S_N(I+\Sigma)=I+(-1)^N\Sigma^{N+1}.
\)
By \eqref{eq:sigma-norm}, $\norm{\Sigma^{N+1}}\to 0$, proving \eqref{eq:resolvent-series} as a two-sided operator inverse.
Set $Y:=(I+\Sigma)^{-1}H$. Proposition~\ref{prop:identity} gives
\[
\begin{aligned}
\A(GH)&=\A(\mathcal PY)=Y+\Sigma Y=H,\\
(GH)(T)&=0.
\end{aligned}
\]
It remains to retain the pointwise half-order kernel after summing the resolvent series. For $n\ge 1$, the time-kernel majorant is nonnegative, so Tonelli's theorem permits the exchange of the $r$- and $s$-integrals; the Beta integral below then shows that the exchanged integral is finite. Together with \eqref{eq:sigma-power}, this yields
\begin{equation}
\begin{aligned}
\norm{\mathcal{P}\Sigma^nH(t)}_{B^{s+1}}
&\le C_PC_\Sigma^n\frac{\pi^{n/2}}{\Gamma(n/2)}B\Bigl(\frac n2,\frac12\Bigr)\int_t^T(s-t)^{(n-1)/2}\norm{H(s)}_{B^s}\,ds\\
&\le C_PC_\Sigma^n\frac{\pi^{(n+1)/2}\bar T^{n/2}}{\Gamma((n+1)/2)}\int_t^T(s-t)^{-1/2}\norm{H(s)}_{B^s}\,ds\\
&=:c_n\int_t^T(s-t)^{-1/2}\norm{H(s)}_{B^s}\,ds,\qquad
c_n:=C_PC_\Sigma^n\frac{\pi^{(n+1)/2}\bar T^{n/2}}{\Gamma((n+1)/2)}.
\end{aligned}\label{eq:green-series-bound}
\end{equation}
For $n=0$ the same estimate holds with $c_0:=C_P$. Moreover, $K_G=\sum_{n\ge 0}c_n<\infty$, and summing $GH=\sum_{n\ge 0}(-1)^n\mathcal{P}\Sigma^nH$ proves \eqref{eq:green-bound}.
\end{proof}
\subsection{Nonzero Terminal Data and the Terminal Propagator}
\begin{proposition}[Terminal propagator]\label{prop:propagator}
Let $\varphi\in B^{s+1}$. For $0\le t<T$, define
\[
q_\varphi(t):=P_{t,T}\varphi,\qquad
d_\varphi(t):=\A q_\varphi(t),\qquad
U(t;T)\varphi:=q_\varphi(t)-Gd_\varphi(t).
\]
Then $q_\varphi$ and $d_\varphi$ extend continuously to $[0,T]$ with
\(
\A[U(\cdot;T)\varphi]=0, U(T;T)\varphi=\varphi.
\)
Moreover,
\[
\norm{d_\varphi}_{C([0,T];B^s)}
\le C_{\mathrm{tdef}}\norm{\varphi}_{B^{s+1}},
\qquad
\norm{U(t;T)\varphi}_{B^{s+1}}
\le C_{\mathrm{ter}}\norm{\varphi}_{B^{s+1}},
\]
where
\(
C_{\mathrm{tdef}}
:=
C_{\mathrm{def}}
\left(
\frac{4}{e\lambda}
+\frac{64}{e^2\lambda^2}
\right),
C_{\mathrm{ter}}
:=
C_P+2K_GC_{\mathrm{tdef}}\sqrt{\bar T}.
\)
\end{proposition}
\begin{proof}
Lemma~\ref{lem:diagonal} and \eqref{eq:pt-terminal} imply $q_\varphi\in C([0,T];B^{s+1})$. For $t<T$, \eqref{eq:adjoint-identity} gives
\[
d_\varphi(t;x)=\int_{\R^d}e^{ix\cdot\xi}R_{t,T}(x;\xi)\hat\varphi(\xi)\,d\xi.
\]
With $q=(T-t)|\xi|^2$, \eqref{eq:defect-symbol-bound} gives
\[
\norm{R_{t,T}(\cdot;\xi)}_{A^s}\le C_{\mathrm{def}}\Bigl(\frac{4}{e\lambda}|\xi|+\frac{4}{e\lambda}+\frac{64}{e^2\lambda^2}\Bigr)\le C_{\mathrm{tdef}}(1+|\xi|).
\]
Lemma~\ref{lem:symbol} therefore gives
\[
\norm{d_\varphi(t)}_{B^s}\le C_{\mathrm{tdef}}\norm{\varphi}_{B^{s+1}}.
\]
Moreover, $R_{t,T}(\cdot;\xi)\to 0$ in $A^s$ as $t\uparrow T$; dominated convergence then gives $d_\varphi(t)\to 0$ in $B^s$, hence $d_\varphi\in C([0,T];B^s)$.
Since
\[
(\A q_\varphi)(t;x)=d_\varphi(t;x),\qquad \bigl(\A(Gd_\varphi)\bigr)(t;x)=d_\varphi(t;x),
\]
we obtain $\bigl(\A[U(\cdot;T)\varphi]\bigr)(t;x)=0$. Also, $P_{t,T}\varphi\to\varphi$ in $B^{s+1}$ and
\[
\norm{Gd_\varphi(t)}_{B^{s+1}}\le 2K_GC_{\mathrm{tdef}}\sqrt{T-t}\,\norm{\varphi}_{B^{s+1}}\to 0,
\]
so $\bigl(U(T;T)\varphi\bigr)(x)=\varphi(x)$. The same estimate gives
\[
\norm{U(t;T)\varphi}_{B^{s+1}}\le\bigl(C_P+2K_GC_{\mathrm{tdef}}\sqrt{T-t}\bigr)\norm{\varphi}_{B^{s+1}}\le C_{\mathrm{ter}}\norm{\varphi}_{B^{s+1}}.
\]
\end{proof}
\subsection{Augmented Green and Terminal Operators}
\begin{proposition}[Augmented Green and terminal operators]\label{prop:augmented}
Let $0<T\le\bar T$, $H\in C([0,T];A^s)$, and $\varphi\in A^{s+1}$. Write
\(
H(t;x)=H_\infty(t)+H_0(t;x),\varphi(x)=\varphi_\infty+\varphi_0(x),
\)
with
\(
H_\infty\in C([0,T]), H_0\in C([0,T];B^s), \varphi_\infty\in\R, \varphi_0\in B^{s+1}.
\)
Define
\[
\begin{aligned}
(G_AH)(t;x)&:=\int_t^TH_\infty(r)\,dr+(GH_0)(t;x),\\
\bigl(U_A(t;T)\varphi\bigr)(x)&:=\varphi_\infty+\bigl(U(t;T)\varphi_0\bigr)(x).
\end{aligned}
\]
Then
\(
G_AH,\ U_A(\cdot;T)\varphi\in C\bigl([0,T];A^{s+1}\bigr)\cap C^1\bigl([0,T];A^{s-1}\bigr),
\)
and
\[
\bigl(\A(G_AH)\bigr)(t;x)=H(t;x),\qquad (G_AH)(T;x)=0,
\]
\[
\bigl(\A[U_A(\cdot;T)\varphi]\bigr)(t;x)=0,\qquad \bigl(U_A(T;T)\varphi\bigr)(x)=\varphi(x).
\]
Moreover,
\[
\norm{G_AH(t)}_{A^{s+1}}\le K_{G,A}\int_t^T(r-t)^{-1/2}\norm{H(r)}_{A^s}\,dr,
\]
\[
\norm{U_A(t;T)\varphi}_{A^{s+1}}\le C_{\mathrm{ter},A}\norm{\varphi}_{A^{s+1}},
\]
where
\(
K_{G,A}:=K_G+\sqrt{\bar T}, C_{\mathrm{ter},A}:=\max\{1,C_{\mathrm{ter}}\}.
\)
\end{proposition}
\begin{proof}
The nonconstant components follow from Propositions~\ref{prop:green}
and \ref{prop:propagator}; the constant components satisfy the scalar
terminal ODE. In the sense of distributions,
\[
\partial_t(G_AH)=-\tfrac12a:D^2(G_AH)-H\in C([0,T];A^{s-1}),
\]
and the corresponding identity holds for $U_A\varphi$ without the source.
The Banach-valued fundamental theorem of calculus gives the stated
$C^1$ regularity, with one-sided derivatives at the endpoints. Finally,
\[
\begin{aligned}
\|G_AH(t)\|_{A^{s+1}}
&\le\int_t^T|H_\infty(r)|\,dr
+K_G\int_t^T(r-t)^{-1/2}\|H_0(r)\|_{B^s}\,dr,\\
\|U_A(t;T)\varphi\|_{A^{s+1}}
&\le|\varphi_\infty|+C_{\mathrm{ter}}\|\varphi_0\|_{B^{s+1}}.
\end{aligned}
\]
Use $1\le\sqrt{\bar T}(r-t)^{-1/2}$ in the first bound.
\end{proof}
\begin{proof}[Proof of Theorem~\ref{thm:linear-wellposedness}]
Proposition~\ref{prop:augmented} gives the representation, regularity, and
bounds. The difference of two solutions with the same data is bounded
and classical and satisfies the homogeneous equation with zero terminal
data. Lemma~\ref{lem:uniqueness}, with $b=0$, gives uniqueness.
\end{proof}
Theorem~\ref{thm:linear-wellposedness} is the linear solver used below.

\section{Successive Approximation for the HJB Equation and Optimal Feedback}
\subsection{The Nonlinear Hamiltonian and the Solution Map}
If $V\in A^{s+1}$, then
\[
\partial_iV\in B^s,\qquad i=1,\ldots,d.
\]
Hence
\[
f\cdot\nabla V,\quad (\nabla V)^{\top}Q\nabla V\in B^s,\qquad F(t;\nabla V;V)\in A^s.
\]
Thus the nonlinear term lies in the source space of Theorem~\ref{thm:linear-wellposedness}.
\begin{lemma}[Nonlinear and feedback estimates]\label{lem:nonlinear}
For every $M>0$ there exist $B_M,L_M,C_u\ge0$, depending only on $M$ and the coefficient bounds, such that whenever $V,W\in C([0,T];A^{s+1})$ satisfy $\norm{V(t)}_{A^{s+1}},\norm{W(t)}_{A^{s+1}}\le M$, then
\begin{equation}
\norm{F(t;\nabla V;V)}_{A^s}\le B_M,\label{eq:nonlinear-bound}
\end{equation}
\begin{equation}
\norm{F(t;\nabla V;V)-F(t;\nabla W;W)}_{A^s}\le L_M\norm{V-W}_{A^{s+1}},\label{eq:nonlinear-lipschitz}
\end{equation}
\begin{equation}
\norm{u_V-u_W}_{B^s}\le C_u\norm{V-W}_{A^{s+1}},\qquad u_V:=-\frac12R^{-1}g^{\top}\nabla_xV,\label{eq:feedback-lipschitz}
\end{equation}
Here $C_R:=|R^{-1}|_1$ and, with coefficient norms in $C([0,\bar T];A^s)$,
\[
\begin{aligned}
B_M&=\|\ell\|+M(|\gamma|+\|f\|)+\tfrac14M^2\|Q\|,\\
L_M&=\|f\|+\tfrac12M\|Q\|+|\gamma|,\qquad
C_u=\tfrac12C_R\|g\|.
\end{aligned}
\]
Moreover, $F(t;\nabla V;V)\in A^s$, the Hamiltonian difference belongs to $A^s$, and $u_V\in B^s(\R^d;\R^m)$.
\end{lemma}
\begin{proof}
Set $C_R:=|R^{-1}|_1$. The componentwise algebra estimate gives
\begin{equation}
Q:=gR^{-1}g^{\top}\in C([0,T];A^s),\qquad \norm{Q(t)}_{A^s}\le C_R\norm{g(t)}_{A^s}^2.\label{eq:q-regularity}
\end{equation}
Since
\(
\max_i\norm{\partial_iV}_{B^s}\le\norm{V}_{A^{s+1}},
\)
the module property gives the first estimate with
\[
B_M:=\norm{\ell}_{C([0,\bar T];A^s)}+M(\norm{f}_{C([0,\bar T];A^s)}+|\gamma|)+\frac{M^2}{4}\norm{Q}_{C([0,\bar T];A^s)}.
\]
In the difference of two Hamiltonians the running cost cancels, and
\[
(\nabla V)^{\top}Q\nabla V-(\nabla W)^{\top}Q\nabla W=(\nabla V-\nabla W)^{\top}Q\nabla V+(\nabla W)^{\top}Q(\nabla V-\nabla W).
\]
Finally, \eqref{eq:nonlinear-lipschitz} and \eqref{eq:feedback-lipschitz} hold with
\(
L_M:=\norm{f}_{L^\infty_tA^s_x}+\frac{M}{2}\norm{Q}_{L^\infty_tA^s_x}+|\gamma|,
\)
\( 
C_u:=\frac12\,C_R\norm{g}_{L^\infty_tA^s_x}.
\)
All contractions use the entrywise norms and introduce no additional dimension factor.
\end{proof}
Define the nonlinear value map
\begin{equation}
\mathcal{T}(V)(t):=U_A(t;T)\varphi+G_A\bigl[F(\cdot;\nabla V;V)\bigr](t).\label{eq:iteration-map}
\end{equation}
Theorem~\ref{thm:linear-wellposedness} and Lemma~\ref{lem:nonlinear} show that $\mathcal{T}$ is well defined on $C([0,T];A^{s+1})$ and that
\begin{equation}
\norm{\mathcal{T}(V)(t)}_{A^{s+1}}\le C_{\mathrm{ter},A}\norm{\varphi}_{A^{s+1}}+K_{G,A}\int_t^T(r-t)^{-1/2}\norm{F(r;\nabla V(r),V(r))}_{A^s}\,dr.\label{eq:iteration-bound}
\end{equation}
The iteration is $V^{n+1}=\mathcal{T}(V^n)$ with $V^0=0$. By Theorem~\ref{thm:linear-wellposedness}, this is equivalent to the terminal-value PDE iteration of Algorithm~\ref{alg:gradient-iteration}; in particular, every iterate is uniquely defined.
For $V,W$ in the same value-function ball, Proposition~\ref{prop:augmented} and \eqref{eq:nonlinear-lipschitz} give
\[
\begin{aligned}
&\norm{\mathcal{T}(V)(t)-\mathcal{T}(W)(t)}_{A^{s+1}}\\
&\le K_{G,A}\int_t^T(r-t)^{-1/2}
\norm{
F(r;\nabla V(r),V(r))
-
F(r;\nabla W(r),W(r))
}_{A^s}\,dr
\\
&\le
K_{G,A}L_M
\int_t^T(r-t)^{-1/2}
\norm{V(r)-W(r)}_{A^{s+1}}\,dr .
\end{aligned}
\]
\subsection{Short-horizon invariant ball}
Set
\(
A_\varphi:=C_{\mathrm{ter},A}\norm{\varphi}_{A^{s+1}}, M:=2A_\varphi+1,
\)
and denote by
\begin{equation}
\mathcal{V}_M(T):=\Bigl\{V\in C\bigl([0,T];A^{s+1}\bigr):
\norm{V}_{C([0,T];A^{s+1})}\le M\Bigr\}.\label{eq:ball}
\end{equation}
On the reference interval $[0,\bar T]$, define
\(
T_0:=\min\Bigl\{\bar T,\ \Bigl(\frac{A_\varphi+1}{2K_{G,A}B_M}\Bigr)^2\Bigr\}.
\)
If $B_M=0$, the second entry in the minimum is understood as $+\infty$.
\begin{proposition}[Short-horizon invariance and uniform iterate bound]\label{prop:invariance}
Let $A_\varphi$, $M$, $B_M$, and $T_0$ be defined above. For every $0<T\le T_0$ ,
\begin{equation}
\mathcal{T}\bigl(\mathcal{V}_M(T)\bigr)\subset\mathcal{V}_M(T),\label{eq:invariance-set}
\end{equation}
\end{proposition}
\begin{proof}
Fix $V\in\mathcal{V}_M(T)$. Lemma~\ref{lem:nonlinear} gives
\(
\norm{F(r;\nabla V(r);V(r))}_{A^s}\le B_M.
\)
For every $t\in[0,T]$, \eqref{eq:iteration-bound} yields
\begin{equation}
\begin{aligned}
\norm{\mathcal{T}(V)(t)}_{A^{s+1}}&\le A_\varphi+K_{G,A}B_M\int_t^T(r-t)^{-1/2}\,dr\\
&\le A_\varphi+2K_{G,A}B_M\sqrt T\\
&\le A_\varphi+(M-A_\varphi)=M.
\end{aligned}\label{eq:invariance-proof}
\end{equation}
The last inequality follows from $T\le T_0$; when $B_M=0$, it is immediate. Theorem~\ref{thm:linear-wellposedness} gives $\mathcal{T}(V)\in C([0,T];A^{s+1})$, so \eqref{eq:invariance-set} follows.
\end{proof}
\subsection{Gamma-factorial convergence}
By Proposition~\ref{prop:invariance}, all iterates lie in $\mathcal{V}_M(T)$ for $0<T\le T_0$, and \eqref{eq:nonlinear-lipschitz} applies uniformly to consecutive iterates.
For $n\ge 0$, set
\(
D_0:=\norm{V^1-V^0}_{C([0,T];A^{s+1})},\)
and
\(
e_n(t):=\norm{V^{n+1}(t)-V^n(t)}_{A^{s+1}}, 
\)
For $n\ge 1$, the two iterates have the same terminal contribution, and therefore
\begin{equation}
V^{n+1}(t)-V^n(t)=G_A\bigl[F(\cdot;\nabla V^n; V^n)-F(\cdot;\nabla V^{n-1};V^{n-1})\bigr](t).\label{eq:error-recursion}
\end{equation}
Then by Proposition~\ref{prop:invariance} and Proposition~\ref{prop:augmented},
\begin{equation}
\begin{aligned}
e_n(t)&\le K_{G,A}\int_t^T(r-t)^{-1/2}\norm{F(r;\nabla V^n(r);V^n(r))-F(r;\nabla V^{n-1}(r);V^{n-1}(r))}_{A^s}\,dr\\
&\le K_{G,A}L_M\int_t^T(r-t)^{-1/2}e_{n-1}(r)\,dr.
\end{aligned}\label{eq:volterra-inequality}
\end{equation}
Inequalities of the form \eqref{eq:volterra-inequality} belong to the classical class of weakly singular Volterra, or Henry--Gronwall, inequalities; see \cite[Lemma 7.1.1]{henry1981}. 
Here and below, $\Gamma$ denotes the Euler Gamma function and
\(
B(\alpha;\beta):=\frac{\Gamma(\alpha)\Gamma(\beta)}{\Gamma(\alpha+\beta)}, \alpha,\beta>0
\)
denotes the Euler Beta function.
\begin{proposition}[Gamma-factorial convergence]\label{prop:gamma}
Let $D_0$ be defined above. Then there exist
$V^{*}\in C([0,T];A^{s+1})$ and constants $C_\Gamma\ge0$, $\Lambda>0$ such that
\(
V^n\to V^{*}\quad\text{in }C\bigl([0,T];A^{s+1}\bigr),
\)
and
\begin{equation}
\norm{V^n-V^{*}}_{C([0,T];A^{s+1})}
\le
C_\Gamma\frac{\Lambda^n}{\Gamma(n/2+1)},
\qquad n\ge0,
\label{eq:gamma-rate}
\end{equation}
where
\(
\Lambda:=\max\{1,K_{G,A}L_M\sqrt{\pi T}\},
C_\Gamma
:=
D_0\bigl(1+2K_{G,A}L_M\sqrt{T}\bigr)
e^{\pi T K_{G,A}^2L_M^2}.
\)
\end{proposition}
\begin{proof}
For $n\ge0$, we claim that
\(
e_n(t)\le
D_0\frac{(K_{G,A}L_M\sqrt\pi)^n(T-t)^{n/2}}
{\Gamma(n/2+1)}.
\)
The case $n=0$ follows from the definition of $D_0$. Assuming the claim
for $n-1$, \eqref{eq:volterra-inequality} gives
\begin{equation}
e_n(t)\le
D_0\frac{(K_{G,A}L_M)^n\pi^{(n-1)/2}}
{\Gamma((n+1)/2)}
\int_t^T(r-t)^{-1/2}(T-r)^{(n-1)/2}\,dr.
\label{eq:induction-step}
\end{equation}
By the change of variables $r=t+(T-t)\theta$ and the Beta identity,
\begin{equation}
\int_t^T(r-t)^{-1/2}(T-r)^{(n-1)/2}\,dr
=
(T-t)^{n/2}
B\Bigl(\frac12,\frac{n+1}{2}\Bigr).
\label{eq:beta-integral}
\end{equation}
Since $\Gamma(1/2)=\sqrt\pi$, the claim follows.
Set
\(
\theta_T:=K_{G,A}L_M\sqrt{\pi T}.
\)
Taking the supremum in time gives
\[
\norm{V^{n+1}-V^n}_{C([0,T];A^{s+1})}
\le
D_0\frac{\theta_T^n}{\Gamma(n/2+1)}.
\]
The series on the right is summable, hence $V^n$ converges in
$C([0,T];A^{s+1})$ to some $V^{*}$.
Moreover, for $m\ge0$,
\[
\frac{\Gamma(n/2+1)}
{\Gamma((n+2m)/2+1)}
\le\frac1{m!},
\qquad
\frac{\Gamma(n/2+1)}
{\Gamma((n+2m+1)/2+1)}
\le\frac{2}{\sqrt\pi\,m!}.
\]
Therefore
\[
\begin{aligned}
\norm{V^n-V^{*}}_{C([0,T];A^{s+1})}
&\le
D_0\sum_{j=n}^{\infty}
\frac{\theta_T^j}{\Gamma(j/2+1)}\\
&\le
D_0\left(1+\frac{2\theta_T}{\sqrt\pi}\right)
e^{\theta_T^2}
\frac{\theta_T^n}{\Gamma(n/2+1)}\\
&\le
C_\Gamma\frac{\Lambda^n}{\Gamma(n/2+1)}.
\end{aligned}
\]
This proves \eqref{eq:gamma-rate}.
For the main theorem, $D_0\le M$ and $T\le\bar T$ allow the choices
\[
\bar\Lambda=\max\{1,K_{G,A}L_M\sqrt{\pi\bar T}\},\qquad
\bar C_\Gamma=M(1+2K_{G,A}L_M\sqrt{\bar T})e^{\pi\bar T K_{G,A}^2L_M^2},
\]
which are independent of $T\le T_0$.
\end{proof}
\begin{proposition}[Limit equation and time regularity]\label{prop:limit}
For $s>2$, let $V^{*}$ be the limit obtained in Proposition~\ref{prop:gamma}. Then
\[
V^{*}\in C^1\bigl([0,T];A^{s-1}\bigr)\cap C\bigl([0,T];A^{s+1}\bigr),
\]
and
\begin{equation}
\partial_tV^{*}+\frac12\,a:D_x^2V^{*}+F(t;\nabla_xV^{*};V^{*})=0,\qquad V^{*}(T)=\varphi.\label{eq:limit-equation}
\end{equation}
Moreover, $V^{*}$ is bounded and belongs to
\(
C^{1,2}\bigl([0,T]\times\R^d\bigr),
\)
so \eqref{eq:limit-equation} holds pointwise.
\end{proposition}

\begin{proof}
Set
\[
G_n:=\frac12\,a:D_x^2V^{n+1}+F(t;\nabla V^n;V^n),
\qquad
G_*:=\frac12\,a:D_x^2V^{*}+F(t;\nabla V^{*};V^{*}).
\]
By Proposition~\ref{prop:gamma}, the derivative and module estimates, and
\eqref{eq:nonlinear-lipschitz},
\[
\begin{aligned}
\norm{G_n-G_*}_{C([0,T];A^{s-1})}
&\le
\frac12\norm{a}_{C([0,T];A^s)}
\norm{V^{n+1}-V^*}_{C([0,T];A^{s+1})} \\
&\quad+
L_M\norm{V^n-V^*}_{C([0,T];A^{s+1})}
\longrightarrow0.
\end{aligned}
\]
Since
\[
V^{n+1}(t)=\varphi+\int_t^TG_n(r)\,dr,
\]
passing to the limit gives
\[
V^*(t)=\varphi+\int_t^TG_*(r)\,dr
\quad\text{in }A^{s-1}.
\]
As $G_*\in C([0,T];A^{s-1})$, the Banach-valued fundamental theorem
of calculus yields
\[
V^*\in C^1([0,T];A^{s-1}),
\qquad
\partial_tV^*=-G_*,
\]
and hence \eqref{eq:limit-equation}.
Finally, $V^*\in C([0,T];A^{s+1})$ implies
\(
\nabla_xV^*\in C([0,T];B^s),
D_x^2V^*\in C([0,T];B^{s-1}).
\)
Since $s>2$, Proposition~\ref{prop:basic} yields
\[
V^*\in C^{1,2}([0,T]\times\R^d).
\]
Hence \eqref{eq:limit-equation} holds pointwise.
\end{proof}
\begin{corollary}[Feedback convergence]\label{cor:feedback}
Let
\(
u^n:=-\frac12R^{-1}g^\top\nabla_xV^n,
u^*:=-\frac12R^{-1}g^\top\nabla_xV^*.
\)
Then
\(
u^n,\ u^*\in C\bigl([0,T];B^s(\R^d;\R^m)\bigr),
\)
and
\begin{equation}
\norm{u^n-u^*}_{C([0,T];B^s)}
\le
C_uC_\Gamma
\frac{\Lambda^n}{\Gamma(n/2+1)},
\qquad n\ge0.
\label{eq:feedback-rate}
\end{equation}
\end{corollary}
\begin{proof}
By Propositions~\ref{prop:gamma}--\ref{prop:limit}, the differentiation
and module estimates give
\(
u^n,u^*\in C([0,T];B^s)
\)
and
\(
\norm{u^n-u^*}_{C([0,T];B^s)}
\le
C_u\norm{V^n-V^*}_{C([0,T];A^{s+1})}.
\)
The conclusion follows from \eqref{eq:gamma-rate}.
\end{proof}
\subsection{Closed-Loop Well-Posedness, Verification, and Optimality}
\begin{proposition}[Verification and optimality]\label{prop:verification}
Fix $(t;x)\in[0,T]\times\R^d$ and consider the closed-loop equation
\begin{equation}
\begin{cases}
dX_r^{*}=\bigl(f(r,X_r^{*})+g(r,X_r^{*})u^{*}(r,X_r^{*})\bigr)\,dr+\sigma(r,X_r^{*})\,dW_r,\\
X_t^{*}=x,
\end{cases}\label{eq:closed-loop-sde}
\end{equation}
the feedback $u^*$ induces a unique strong solution $X^*$ of the closed-loop state equation, and
\(
u_r^*:=u^*(r,X_r^*)\in\mathcal U_{t,T}(x),
V^*=V_{\mathrm{val}},
J(t;x;u^*)=V^*(t;x),
\)
and $u^*$ is optimal.
\end{proposition}
\begin{proof}
We first verify that the feedback control induces a well-posed closed-loop state equation. By Corollary~\ref{cor:feedback},
\(
u^{*}\in C\bigl([0,T];B^s(\R^d;\R^m)\bigr).
\)
By Assumption~\ref{ass:standing}, Corollary~\ref{cor:feedback}, and
Propositions~\ref{prop:basic}--\ref{prop:algebra}, since $s>2$,
$f$, $g$, $\sigma$, and $u^*$ are bounded with bounded first spatial
derivatives, uniformly in time. Hence
\(
b^*:=f+gu^*
\)
and $\sigma$ are bounded and globally Lipschitz in $x$, uniformly in $t$.
The standard SDE existence--uniqueness theorem
\cite[Chapter~1, Theorem~6.3]{yong1999} therefore gives a unique strong
solution $X^*$ of \eqref{eq:closed-loop-sde}.
Since $b^*$ and $\sigma$ are bounded, the integral form of the SDE and the Burkholder--Davis--Gundy inequality also give
\[
\E\Bigl[\sup_{t\le r\le T}|X_r^*|^2\Bigr]\le C(1+|x|^2)<\infty.
\]
We next verify admissibility of the induced feedback control
\(
u_r^{*}:=u^{*}(r,X_r^{*}).
\)
Since $u^{*}$ is continuous and $X^{*}$ is adapted with continuous paths, $(u_r^{*})_{r\in[t,T]}$ is progressively measurable. Since $u^{*}$ is bounded,
\[
\E\int_t^T|u_r^{*}|^2\,dr\le T\norm{u^{*}}_{L^\infty_{t,x}}^2<\infty.
\]
Furthermore, $\ell$ and $\varphi$ are bounded by the Barron embedding. Since $R$ is fixed, there exists $C_{R,+}=\|R\|_{\mathrm{op}}$ such that
\[
|v^{\top}Rv|\le C_{R,+}|v|^2.
\]
Consequently, Proposition~\ref{prop:basic}--\ref{prop:algebra} imply that, for any fixed $\gamma\in\mathbb{R}$,
\[
\begin{aligned}
|J(t;x;u^{*})|
&\le
\E\int_t^T e^{-\gamma(r-t)}
\Bigl(
|\ell(r,X_r^*)|
+
C_{R,+}|u^*(r,X_r^*)|^2
\Bigr)\,dr
+
e^{-\gamma(T-t)}
\E|\varphi(X_T^*)|
\\
&\le
e^{|\gamma| T}(T\norm{\ell}_{L^\infty_{t,x}}
+
T C_{R,+}\norm{u^{*}}_{L^\infty_{t,x}}^2
+
\norm{\varphi}_{L^\infty_x})
\\
&\le
e^{|\gamma| T}(T\norm{\ell}_{C([0,T];A^{s})}
+
T C_{R,+}\norm{u^{*}}_{C([0,T];B^{s})}^2
+
\norm{\varphi}_{A^{s+1}})
<\infty.
\end{aligned}
\]
Hence
\(
u^{*}\in\mathcal{U}_{t,T}(x).
\)

Let $u\in\mathcal U_{t,T}(x)$ be arbitrary and write
\[
X_r:=X_r^{t,x;u},
\qquad
p_r:=\nabla_xV^*(r,X_r).
\]
The boundedness of $\sigma$ and $\nabla_xV^*$ gives
\[
\E\int_t^T
e^{-2\gamma(r-t)}
|\sigma(r,X_r)^\top p_r|^2\,dr
<\infty,
\]
so the stochastic integral arising from It\^o's formula applied to
\(
e^{-\gamma(r-t)}V^*(r,X_r)
\)
is a square-integrable martingale with zero expectation.
All drift and cost terms below are integrable because the data and
derivatives of $V^*$ are bounded and
\(
\E\int_t^T|u_r|^2\,dr<\infty.
\)

Applying It\^o's formula and using $V^*(T,\cdot)=\varphi$, we obtain
\[
\begin{aligned}
&J(t;x;u)-V^*(t;x)\\
&=
\E\int_t^T
e^{-\gamma(r-t)}
\Bigl[
\ell(r,X_r)+u_r^\top Ru_r
+\partial_tV^*(r,X_r)
+\frac12a(r,X_r):D_x^2V^*(r,X_r)
\\
&\qquad
+\bigl(f(r,X_r)+g(r,X_r)u_r\bigr)\cdot p_r
-\gamma V^*(r,X_r)
\Bigr]\,dr
\\
&=
\E\int_t^T
e^{-\gamma(r-t)}
\left[
u_r^\top Ru_r
+\bigl(g(r,X_r)u_r\bigr)\cdot p_r
+\frac14p_r^\top Q(r,X_r)p_r
\right]\,dr
\\
&=
\E\int_t^T
e^{-\gamma(r-t)}
\bigl(u_r-u^*(r,X_r)\bigr)^\top
R
\bigl(u_r-u^*(r,X_r)\bigr)\,dr\ge0.
\end{aligned}
\]
Here the first equality follows from It\^o's formula applied to
$e^{-\gamma(r-t)}V^*(r,X_r)$ and the terminal condition, the second
from the discounted HJB equation, and the third from
$u^*=-\frac12R^{-1}g^\top p$ and $Q=gR^{-1}g^\top$.

For the admissible closed-loop control
$u_r=u^*(r,X_r^*)$, the right-hand side vanishes. Hence
\[
J(t;x;u^*)=V^*(t;x)=V_{\mathrm{val}}(t;x),
\]
and $u^*$ is optimal.
\end{proof}
\begin{proof}[Proof of Theorem~\ref{thm:solvability}]
Propositions~\ref{prop:invariance}, \ref{prop:gamma}, and
\ref{prop:limit} yield a bounded classical solution
\[
V^*\in C^1([0,T];A^{s-1})\cap C([0,T];A^{s+1}),
\]
with
\[
\norm{V^n-V^*}_{C([0,T];A^{s+1})}
\le
C_\Gamma\frac{\Lambda^n}{\Gamma(n/2+1)}.
\]
Together with Corollary~\ref{cor:feedback}, this gives, after enlarging
the constant if necessary, the convergence estimate
\eqref{eq:main-convergence}. Proposition~\ref{prop:verification} further
yields
\[
J(t;x;u^*)=V^*(t;x)=V_{\mathrm{val}}(t;x),
\]
so $u^*$ is admissible and optimal.
It remains only to prove uniqueness. Let $W$ be another bounded classical
solution in
\[
C^1([0,T];A^{s-1})\cap C([0,T];A^{s+1}),
\qquad W(T)=\varphi,
\]
and set $Z:=W-V^*$. Subtracting the two HJB equations and using
$Q=Q^\top$ gives
\[
\partial_tZ+\frac12a:D_x^2Z
+\left[
f-\frac14Q\bigl(\nabla_xW+\nabla_xV^*\bigr)
\right]\cdot\nabla_xZ
-\gamma Z=0,
\qquad
Z(T)=0.
\]
Set
\(
b:=f-\frac14Q\bigl(\nabla_xW+\nabla_xV^*\bigr).
\)
By the Barron regularity, $b$ is bounded and continuous and $Z$ is a
bounded classical solution. Define
\[
\widetilde Z(t,x):=e^{\gamma(T-t)}Z(t,x).
\]
Then $\widetilde Z$ is bounded and classical on
$[0,T]\times\R^d$, satisfies
\[
\partial_t\widetilde Z
+\frac12a:D_x^2\widetilde Z
+b\cdot\nabla_x\widetilde Z=0,
\qquad
\widetilde Z(T)=0.
\]
Hence Lemma~\ref{lem:uniqueness} yields
$\widetilde Z\equiv0$, and therefore $Z\equiv0$.
Thus $W=V^*$.
\end{proof}
\section{Joint Space--Time Neural Approximation}
The approximation is obtained in two steps. We first approximate each
iterate in time by a finite cosine expansion and then approximate its
spatial coefficients by spectral Barron cosine networks.
For brevity, write
\[
\|v\|_{CX}:=\|v\|_{C([0,T];X)}.
\]
We use the classical Fej\'er--Jackson construction for trigonometric
approximation; see, e.g., \cite{devore1993,anastassiou2007}.
For $N_t \in \mathbb N$ , let $m:=\left\lfloor\frac{N_t}{2}\right\rfloor+1,$
\[
F_m(\eta)
:=
\frac1m
\left(
\frac{\sin(m\eta/2)}{\sin(\eta/2)}
\right)^2
=
\sum_{j\in\mathbb Z}q_{j,m}e^{ij\eta},
\qquad
q_{j,m}:=
\left(1-\frac{|j|}{m}\right)_+,
\]
be the Fej\'er kernel. We use the associated Jackson kernel
\begin{equation}
K_m(\eta)
:=
\frac{3}{2\pi m(2m^2+1)}
\left(
\frac{\sin(m\eta/2)}{\sin(\eta/2)}
\right)^4.
\label{eq:jackson-kernel}
\end{equation}
The form \eqref{eq:jackson-kernel} will be used for the approximation
error. For the coefficient estimates, using
\[
\left(
\frac{\sin(m\eta/2)}{\sin(\eta/2)}
\right)^4
=
m^2F_m(\eta)^2
\]
gives the finite Fourier expansion
\begin{equation}
K_m(\eta)
=
\frac1{2\pi}
\sum_{|k|\le2m-2}\mu_{k,m}e^{ik\eta},
\qquad
\mu_{k,m}
:=
\frac{3m}{2m^2+1}
\sum_{j\in\mathbb Z}q_{j,m}q_{k-j,m}.
\label{eq:jackson-fourier}
\end{equation}
Since
\(
\sum_jq_{j,m}^2=\frac{2m^2+1}{3m},
\)
Cauchy--Schwarz gives
\[
\mu_{0,m}=1,
\qquad
0\le\mu_{k,m}\le1.
\]
Thus $K_m$ is nonnegative and normalized.

\begin{proposition}[Banach-valued temporal Jackson approximation]
\label{prop:jackson}
Let $X_1\hookrightarrow X_0$ be Banach spaces and
\(
v\in C([0,T];X_1)\cap C^1([0,T];X_0).
\)
For every integer $N_t\ge1$, there exists
\[
P_{N_t}v(t)
=
\sum_{k=0}^{N_t}
B_k\cos\Bigl(\frac{k\pi t}{T}\Bigr),
\qquad B_k\in X_1,
\]
such that, with
\(
M:=\|v\|_{C X_1},
C_J:=\frac{3\pi^3}{2},
\)
\begin{equation}
\|v-P_{N_t}v\|_{C X_0}
\le
\frac{2C_JT}{\pi N_t}
\|\partial_t v\|_{C X_0},
\label{eq:scaled-jackson}
\end{equation}
and
\begin{equation}
\|B_0\|_{X_1}\le M,
\qquad
\|B_k\|_{X_1}\le2M,
\qquad
\|B_k\|_{X_0}
\le
\frac{2T}{k\pi}\|\partial_t v\|_{C X_0},
\quad 1\le k\le N_t.
\label{eq:jackson-coefficient-bounds}
\end{equation}
\end{proposition}

\begin{proof}
Let
\[
w(\theta):=v(T|\theta|/\pi),
\qquad -\pi\le\theta\le\pi,
\]
and extend $w$ evenly and $2\pi$-periodically. Set 
\[
m:=\left\lfloor\frac{N_t}{2}\right\rfloor+1,
\qquad
p_m(\theta)
:=
(K_m*w)(\theta)
=
\int_{-\pi}^{\pi}K_m(\eta)w(\theta-\eta)\,d\eta,
\]
and
\(
P_{N_t}v(t):=
p_m\Bigl(\frac{\pi t}{T}\Bigr).
\)

Since $w$ is Lipschitz in $X_0$ with constant
$T/\pi\|\partial_t v\|_{C X_0}$,
\[
\|p_m(\theta)-w(\theta)\|_{X_0}
\le
\frac{T}{\pi}\|\partial_t v\|_{C X_0}
\int_{-\pi}^{\pi}|\eta|K_m(\eta)\,d\eta.
\]
For $|\eta|\le\pi$, the elementary bounds
\[
|\sin(m\eta/2)|\le \frac{m|\eta|}{2},
\qquad
|\sin(\eta/2)|\ge \frac{|\eta|}{\pi},
\]
together with $2m^2+1\ge2m^2$, give
\(
K_m(\eta)
\le
\frac{3\pi^3}{4}\,m.
\)
On the other hand, using
\[
|\sin(m\eta/2)|\le1,
\qquad
|\sin(\eta/2)|\ge\frac{|\eta|}{\pi},
\]
we obtain
\[
K_m(\eta)
\le
\frac{3\pi^3}{4}\,
m^{-3}|\eta|^{-4}.
\]
Combining the two bounds gives
\begin{equation}
K_m(\eta)
\le
\frac{3\pi^3}{4}
\min\{m,m^{-3}|\eta|^{-4}\}.
\label{eq:jackson-pointwise-bound}
\end{equation}

Splitting the integral at $|\eta|=m^{-1}$, we obtain
\[
\begin{aligned}
\int_{-\pi}^{\pi}|\eta|K_m(\eta)\,d\eta
&\le
\frac{3\pi^3}{2}
\left[
m\int_0^{1/m}\eta\,d\eta
+
m^{-3}\int_{1/m}^{\pi}\eta^{-3}\,d\eta
\right]
\\
&=
\frac{3\pi^3}{2}
\left[
\frac{1}{2m}
+
\frac{1}{2m}
-\frac{1}{2\pi^2m^3}
\right]
\\
&\le
\frac{3\pi^3}{2m}
=
\frac{C_J}{m}.
\end{aligned}
\]
Since $m^{-1}\le2N_t^{-1}$, this proves
\eqref{eq:scaled-jackson}.

It remains to identify the coefficients. Let
\[
a_0
:=
\frac1{2\pi}\int_{-\pi}^{\pi}w(\theta)\,d\theta,
\qquad
a_k
:=
\frac1\pi
\int_{-\pi}^{\pi}w(\theta)\cos(k\theta)\,d\theta,
\quad k\ge1.
\]
Substituting \eqref{eq:jackson-fourier} into the convolution and using
the evenness of $w$ gives
\begin{equation}
p_m(\theta)
=
a_0+
\sum_{k=1}^{2m-2}
\mu_{k,m}a_k\cos(k\theta).
\label{eq:jackson-convolution-expansion}
\end{equation}
Hence
\(
B_0=a_0,
B_k=\mu_{k,m}a_k
\quad(1\le k\le2m-2),
\)
with $B_k=0$ for $2m-2<k\le N_t$.

Clearly,
\(
\|a_0\|_{X_1}\le M,
\|a_k\|_{X_1}\le2M.
\)
Together with $0\le\mu_{k,m}\le1$, this gives the first two bounds in
\eqref{eq:jackson-coefficient-bounds}.

Finally, by integration by parts in $X_0$,
\[
\begin{aligned}
a_k
&=
\frac2T
\int_0^T
v(t)\cos\Bigl(\frac{k\pi t}{T}\Bigr)\,dt
\\
&=
\frac{2}{k\pi}
\left[
v(t)\sin\Bigl(\frac{k\pi t}{T}\Bigr)
\right]_{t=0}^{t=T}
-
\frac{2}{k\pi}
\int_0^T
\partial_t v(t)
\sin\Bigl(\frac{k\pi t}{T}\Bigr)\,dt
\\
&=
-\frac{2}{k\pi}
\int_0^T
\partial_t v(t)
\sin\Bigl(\frac{k\pi t}{T}\Bigr)\,dt,
\end{aligned}
\]
since $\sin(0)=\sin(k\pi)=0$.
and therefore
\(
\|a_k\|_{X_0}
\le
\frac{2T}{k\pi}
\|\partial_t v\|_{C X_0}.
\)
Since $B_k=\mu_{k,m}a_k$ and $|\mu_{k,m}|\le1$, the last estimate in
\eqref{eq:jackson-coefficient-bounds} follows.
\end{proof}

\begin{proposition}[Approximation of the iterates]\label{prop:iterate-approx}
For every compact $K\subset\R^d$ and $n\ge0$, $N_t,N_x\ge1$,
there exists a shallow cosine network
$\bar V^{n,N_t,N_x}$ such that
\(
Q_V\le(2N_t+1)(N_x+1)
\)
and
\[
\sup_{t\in[0,T]}
\norm{V^n(t)-\bar V^{n,N_t,N_x}(t)}_{L^2(K)}
\le
|K|^{1/2} C_V^{\mathrm{it}}
\left(
N_t^{-1}
+
(1+\log(N_t+1))N_x^{-1/2}
\right),
\]
where $ C_V^{\mathrm{it}} = \max\left\{ \frac{2C_JT}{\pi}C_{\mathrm{time}}, M+\frac{2T}{\pi}C_{\mathrm{time}} \right\}. $

Under \eqref{eq:g-regularity}, the same conclusion holds for $u^n$,
with
\(
Q_u\le(2N_t+1)N_x,
\)$ C_u^{\mathrm{it}} = \max\left\{ \frac{2C_JT}{\pi}C_{\mathrm{time},u}, M_u+\frac{2T}{\pi}C_{\mathrm{time},u} \right\}. $
\end{proposition}
\begin{proof}
By the uniform estimates obtained in Section~4,
\[
\sup_{n\ge0}\norm{V^n}_{C([0,T];A^{s+1})}\le M,
\qquad
\sup_{n\ge0}\norm{\partial_tV^n}_{C([0,T];A^{s-1})}
\le C_{\mathrm{time}}.
\]
For $n\ge1$ the iteration equation gives
$\partial_tV^n=-\tfrac12a:D_x^2V^n-F(t;\nabla V^{n-1};V^{n-1})$.
The derivative and module estimates therefore allow
$C_{\mathrm{time}}:=\tfrac12\Lambda_a M+B_M$.
For $n=0$, the bound holds because $V^0=0$.
Applying Proposition~\ref{prop:jackson} with
\(
X_1=A^{s+1},
X_0=A^{s-1},
\)
gives
\(
P_{N_t}V^n(t)
=
\sum_{k=0}^{N_t}
B_k^n\cos\Bigl(\frac{k\pi t}{T}\Bigr).
\)
By \eqref{eq:jackson-coefficient-bounds},
\[
\|B_0^n\|_{A^{s+1}}\le M,
\qquad
\|B_k^n\|_{A^{s+1}}\le2M,
\qquad
\|B_k^n\|_{A^{s-1}}
\le\frac{2TC_{\mathrm{time}}}{k\pi},
\quad 1\le k\le N_t.
\]
Meanwhile, \eqref{eq:scaled-jackson} yields
\[
\sup_{t\in[0,T]}
\|V^n(t)-P_{N_t}V^n(t)\|_{A^{s-1}}
\le CN_t^{-1}.
\]
Write $B_k^n=c_k^n+h_k^n$ with $h_k^n\in B^{s+1}$.
Then $\|h_0^n\|_{B^0}\le M$, while, since $s>2$,
\eqref{eq:jackson-coefficient-bounds} and $A^{s-1}\hookrightarrow A^0$ give
\[
\|h_k^n\|_{B^0}
\le \frac{2TC_{\mathrm{time}}}{k\pi},
\qquad k\ge1.
\]
Applying Lemma~\ref{thm:cosine}, for each temporal mode there exists an $N_x$-neuron
cosine network $h_k^{n,N_x}$ such that
\[
\norm{h_k^n-h_k^{n,N_x}}_{L^2(K)}
\le
|K|^{1/2}\|h_k^n\|_{B^0}N_x^{-1/2}.
\]
Setting
\(
\Psi_k^{n,N_x}:=c_k^n+h_k^{n,N_x},
\bar V^{n,N_t,N_x}(t)
:=
\sum_{k=0}^{N_t}
\Psi_k^{n,N_x}
\cos\Bigl(\frac{k\pi t}{T}\Bigr),
\)
we obtain
\[
\sup_{t\in[0,T]}
\norm{P_{N_t}V^n(t)-\bar V^{n,N_t,N_x}(t)}_{L^2(K)}
\le|K|^{1/2} C_{V}^{it}(1+\log(N_t+1))N_x^{-1/2}.
\]
Since $A^{s-1}\hookrightarrow L^\infty(\R^d)$,  the triangle inequality yields
\[
\sup_{t\in[0,T]}
\norm{V^n(t)-\bar V^{n,N_t,N_x}(t)}_{L^2(K)}
\le|K|^{1/2}C_{V}^{it}\bigl(N_t^{-1}+(1+\log(N_t+1))N_x^{-1/2}\bigr).
\]
Finally, the product-to-sum identity
\(
\cos a\cos b=\frac12\bigl(\cos(a+b)+\cos(a-b)\bigr)
\)
shows that $\bar V^{n,N_t,N_x}$ is a shallow cosine network with
\(
Q_V\le(2N_t+1)(N_x+1).
\)
Under \eqref{eq:g-regularity}, differentiation of
$u^n=-\frac12R^{-1}g^\top\nabla V^n$ gives
\[
\partial_tu^n=-\frac12R^{-1}
\bigl[(\partial_tg)^\top\nabla V^n+g^\top\nabla(\partial_tV^n)\bigr],
\]
and hence, for finite constants $M_u$ and $C_{\mathrm{time},u}$ independent of $n$,
\[
\sup_n\norm{u^n}_{C([0,T];B^s)}\le M_u,
\qquad
\sup_n\norm{\partial_tu^n}_{C([0,T];B^{s-2})}
\le C_{\mathrm{time},u}.
\]
Apply Proposition~\ref{prop:jackson} to $u^n$. If its vector-valued
temporal coefficients are denoted by $D_k^n$, then
Proposition~\ref{prop:jackson} and $B^{s-2}\hookrightarrow B^0$
give
\[
\|D_0^n\|_{B^0}\le M_u,
\qquad
\|D_k^n\|_{B^0}
\le \frac{2TC_{\mathrm{time},u}}{k\pi}\quad(k\ge1).
\]
The preceding bounds also show that
\(
u^n\in
C([0,T];B^s)
\cap
C^1([0,T];B^{s-2})
\)
uniformly in $n$. Let
\(
P_{N_t}u^n(t)
=
\sum_{k=0}^{N_t}
D_k^n\cos\Bigl(\frac{k\pi t}{T}\Bigr)
\)
be the Jackson polynomial from Proposition~\ref{prop:jackson}.
Then \eqref{eq:scaled-jackson} gives
\[
\sup_{t\in[0,T]}
\|u^n(t)-P_{N_t}u^n(t)\|_{B^{s-2}}
\le
CN_t^{-1}.
\]
Applying Lemma~\ref{thm:cosine} once to each coefficient and summing as above yields
\[
\sup_{t\in[0,T]}
\norm{u^n(t)-\bar u^{n,N_t,N_x}(t)}_{L^2(K;\R^m)}
\le
|K|^{1/2}C_u^{it}\bigl(N_t^{-1}+(1+\log(N_t+1))N_x^{-1/2}\bigr),
\]
with
\(
Q_u\le(2N_t+1)N_x.
\)
\end{proof}
\begin{proof}[Proof of Theorem~\ref{thm:approx}]
Proposition~\ref{prop:iterate-approx} provides the two networks and their
approximation bounds. The embeddings into $L^\infty$, the Gamma estimate,
and the triangle inequality give
\[
\|V^*-\bar V^{n,N_t,N_x}\|_{C L^2(K)}
\le |K|^{1/2}C_\Gamma\frac{\Lambda^n}{\Gamma(n/2+1)}
 +|K|^{1/2}C_{V}^{it}\bigl(N_t^{-1}+(1+\log(N_t+1))N_x^{-1/2}\bigr),
\]
\[
\|u^*-\bar u^{n,N_t,N_x}\|_{C L^2(K)}
\le |K|^{1/2}C_uC_\Gamma\frac{\Lambda^n}{\Gamma(n/2+1)}
 +|K|^{1/2}C_u^{it}\bigl(N_t^{-1}+(1+\log(N_t+1))N_x^{-1/2}\bigr)
\]
where $ C_V=\max\{C_\Gamma,C_V^{\mathrm{it}}\},  C_u^{\mathrm{app}} = \max\{C_uC_\Gamma,C_u^{\mathrm{it}}\}. $ For $n\ge1$, $V^n(T)=\varphi$, so the terminal bound follows directly from the iterate approximation, without an iteration-error term.
\end{proof}

\appendix
\section{Technical proofs and auxiliary estimates}
\subsection{Supplementary proofs}
\begin{proof}[Proof of Proposition~\ref{prop:algebra}]
For $a_j=c_j+h_j$, expand
$a_1a_2=c_1c_2+c_1h_2+c_2h_1+h_1h_2$ and apply the $B^s$ product bound.
The module estimate is the same expansion with one constant equal to zero.
Differentiation kills the constant component, and
$|\xi_i|,\ |\xi_i\xi_j|\le(1+|\xi|),\ (1+|\xi|)^2$, respectively,
give the derivative bounds.
\end{proof}

\begin{proof}[Proof of Proposition~\ref{prop:diffusion}]
The algebra estimate gives $\|\sigma\sigma^\top\|_{A^{s+2}}\le\|\sigma\|_{A^{s+2}}^2$.
Expanding $a(t)-a(r)$ gives time continuity.
Finally, $z^\top a(t,x)z\le |z|^2\sum_{ij}\|a_{ij}(t)\|_\infty
\le\Lambda_a|z|^2$; the lower bound is the ellipticity assumption.
\end{proof}

\begin{proof}[Proof of the remaining estimates in Proposition~\ref{prop:parametrix}]
\mbox{}\par The symbol lemma and Gaussian bound give
\[
\|P_{t,r}h\|_{B^v}\le C_M\int e^{-(\lambda/4)(r-t)|\xi|^2}
|\widehat h(\xi)|(1+|\xi|)^v\,d\xi.
\]
Take $v=s$ for \eqref{eq:pt-bound}, and $v=s+1$, $h=\varphi$ for
\eqref{eq:pt-terminal}. For smoothing, use
\[
(1+|\xi|)e^{-(\lambda/4)(r-t)|\xi|^2}
\le(r-t)^{-1/2}\left[\sqrt{\bar T}+\sqrt{\frac{2}{e\lambda}}\right].
\]
This proves \eqref{eq:pt-smoothing} with the stated $C_P$.
The weighted-majorant formulation of the symbol lemma applies even when
$h$ is initially only in $B^s$; equivalently, one may first truncate
$\widehat h$ and pass to the limit.
\end{proof}
\begin{proof}[Proof of Lemma~\ref{lem:diagonal}]
Fix $0\le v\le s+2$ and $h\in B^v$. The embedding
$A^{s+2}\hookrightarrow A^v$ permits the Gaussian algebra bound at
exponent $v$. The identity
\[
e^{-q\vartheta}-1=-q\vartheta\int_0^1e^{-\theta q\vartheta}\,d\theta
\]
and the uniform bound $\|e^{-q\vartheta}\|_{A^v}\le C_M$ give
\[
\|M_{t,t+\varepsilon}(\cdot;\xi)-1\|_{A^v}
\le(C_M+1)\min\{\Lambda_a\varepsilon|\xi|^2,1\}.
\]
The symbol lemma yields
\[
\sup_{0\le t\le T-\varepsilon}
\|(P_{t,t+\varepsilon}-I)h\|_{B^v}
\le(C_M+1)\int\min\{\Lambda_a\varepsilon|\xi|^2,1\}
|\widehat h(\xi)|(1+|\xi|)^v\,d\xi\longrightarrow0
\]
by dominated convergence. The case $\xi=0$ is immediate.
\end{proof}

\subsection{Wiener--L\'evy and GRS Lemmas}
We record the two standard ingredients used in the Wiener--GRS step of the proof in Section 3.3.
\begin{lemma}[Classical Wiener--L\'evy inversion]\label{lem:wienerlevy}
Let $\psi=c+h\in A^0(\R^d;\C)$. If
\[
\inf_{x\in\R^d}|\psi(x)|>0,
\]
then
\[
\psi^{-1}\in A^0(\R^d;\C).
\]
\end{lemma}
\begin{proof}
Since $h\in C_0$, the lower bound forces $c\neq 0$. Set $b=c^{-1}h$. Since $b\in C_0(\R^d)$,
\[
K:=\overline{b(\R^d)}
\]
is compact and contains $0$. Moreover, the lower bound on $|\psi|=|c|\,|1+b|$ implies that $-1\notin K$. Hence
\[
\Phi(z):=(1+z)^{-1}-1
\]
is holomorphic on an open neighborhood of $K$ and satisfies $\Phi(0)=0$. The classical Wiener--L\'evy theorem \cite[Chapter~6]{rudin1962}, in its
Fourier--$L^1$ form on locally compact abelian groups
\cite[Theorem~2 and the subsequent discussion]{favorov2022}, gives $\Phi(b)\in B^0$. Consequently, $\psi^{-1}=c^{-1}\bigl(1+\Phi(b)\bigr)\in A^0$.
\end{proof}
\begin{lemma}[GRS spectral invariance]\label{lem:grs}
Let $\rho\ge 0$ and
\(
w_\rho(\xi):=(1+|\xi|)^\rho.
\)
For every $u\in L^1_{w_\rho}(\R^d)$,
\(
\operatorname{spec}_{L^1}(u)=\operatorname{spec}_{L^1_{w_\rho}}(u),
\)
where both spectra are computed in the corresponding unitizations.
\end{lemma}
\begin{proof}
The additive group $(\R^d,+)$ is locally compact and is generated by the compact unit ball; moreover, the Haar measure of its $n$-fold sum grows like $n^d$, so the group has polynomial growth. The weight $w_\rho$ is continuous, takes values in $[1,\infty)$, is symmetric, and is submultiplicative because
\[
1+|\xi+\eta|\le(1+|\xi|)(1+|\eta|).
\]
Finally, since the group power of $\xi$ is the sum $n\xi$,
\[
\lim_{n\to\infty}w_\rho(n\xi)^{1/n}=\lim_{n\to\infty}(1+n|\xi|)^{\rho/n}=1.
\]
Thus \cite[Theorem 1.3(iv)]{fendler2006} gives the stated spectral identity in the corresponding unitizations.
\end{proof}
\subsection{Whole-Space Parabolic Uniqueness}
\begin{lemma}[Whole-space parabolic uniqueness by the maximum principle]\label{lem:uniqueness}
Let
\[
a=a^{\top}\in C\bigl([0,T]\times\R^d;\R^{d\times d}\bigr),\qquad b\in C\bigl([0,T]\times\R^d;\R^d\bigr),
\]
and assume that $a$ and $b$ are bounded and that $a$ is uniformly elliptic. If
\[
z\in C^{1,2}\bigl([0,T]\times\R^d\bigr)\cap L^\infty\bigl([0,T]\times\R^d\bigr)
\]
satisfies
\begin{equation}
\partial_tz(t;x)+\frac12\,a(t;x):D_x^2z(t;x)+b(t;x)\cdot\nabla_xz(t;x)=0,\qquad z(T;x)=0,\label{eq:parabolic-uniqueness}
\end{equation}
then
\(
z\equiv 0\,\text{on }[0,T]\times\R^d.
\)
\end{lemma}
\begin{proof}
Reverse time: $w(\tau,x)=z(T-\tau,x)$ solves
$\mathscr Lw=0$, $w(0)=0$, where
$\mathscr L=\partial_\tau-\tfrac12\widetilde a:D^2-\widetilde b\cdot\nabla$.
For $C>\|\operatorname{tr}a\|_\infty+\|b\|_\infty$, the function
$\Psi=e^{C\tau}(1+|x|^2)$ satisfies
\[
\mathscr L\Psi=e^{C\tau}
[C(1+|x|^2)-\operatorname{tr}\widetilde a-2\widetilde b\cdot x]>0.
\]
For fixed $\varepsilon>0$ and all sufficiently large $R$,
$w-\varepsilon\Psi<0$ on the parabolic boundary of $[0,T]\times B_R$.
The maximum principle gives $w\le\varepsilon\Psi$ there.
Let $R\to\infty$ and then $\varepsilon\downarrow0$; repeat for $-w$.
\end{proof}
\subsection{Approximation tool}
\begin{lemma}[Vector Barron approximation]
\label{thm:cosine}
Let $h\in B^0(\R^d;\R^m)$ be real-valued, let $K\subset\R^d$ be compact,
and let $N\ge1$. There is an $N$-neuron cosine network $h_N$ with real
vector output coefficients such that
\[
\|h-h_N\|_{L^2(K;\R^m)}
\le |K|^{1/2}\|h\|_{B^0(\R^d;\R^m)}N^{-1/2},
\]
where the target $L^2$ norm uses the Euclidean vector norm.
\end{lemma}
\begin{proof}
Use Hilbert-space sampling \cite[Proposition~3.1]{feng2026}. Set $S=\sum_j\int|\widehat h_j(\xi)|\,d\xi$;
if $S=0$, take the zero network. Otherwise sample $(j,\xi)$ with density
$|\widehat h_j(\xi)|/S$ and use the real atom
$A(x)=S e_j\cos(\xi\cdot x+\arg\widehat h_j(\xi))$.
Then $\mathbb EA=h$ and $\|A\|_{L^2(K)}\le S|K|^{1/2}$.
For $N$ independent atoms,
$\mathbb E\|N^{-1}\sum_{q=1}^N A_q-h\|_{L^2(K)}^2
\le S^2|K|/N$. One realization gives the result.
\end{proof}
\section*{Acknowledgments}
The authors used artificial intelligence tools to assist with language polishing, translation, and the derivation of some routine proof steps. All mathematical arguments and the final manuscript were independently reviewed and verified by the authors. The authors assume responsibility for all content.

\bibliographystyle{plainnat}
\bibliography{reference}

@article{anastassiou2007,
  author  = {Anastassiou, George A. and Gal, Sorin G.},
  title   = {On the best approximation of vector valued functions by polynomials with coefficients in vector spaces},
  journal = {Annali di Matematica Pura ed Applicata},
  volume  = {186},
  number  = {2},
  pages   = {251--265},
  year    = {2007},
  doi     = {10.1007/s10231-006-0003-4}
}

@article{aronson1967,
  author  = {Aronson, Donald G.},
  title   = {Bounds for the fundamental solution of a parabolic equation},
  journal = {Bulletin of the American Mathematical Society},
  volume  = {73},
  number  = {6},
  pages   = {890--896},
  year    = {1967},
  doi     = {10.1090/S0002-9904-1967-11830-5}
}

@article{barron1993,
  author  = {Barron, Andrew R.},
  title   = {Universal approximation bounds for superpositions of a sigmoidal function},
  journal = {IEEE Transactions on Information Theory},
  volume  = {39},
  number  = {3},
  pages   = {930--945},
  year    = {1993},
  doi     = {10.1109/18.256500}
}

@article{chen2023,
  author  = {Chen, Ziang and Lu, Jianfeng and Lu, Yulong and Zhou, Shengxuan},
  title   = {A regularity theory for static {Schr\"odinger} equations on $\mathbb{R}^d$ in spectral {Barron} spaces},
  journal = {SIAM Journal on Mathematical Analysis},
  volume  = {55},
  number  = {1},
  pages   = {557--570},
  year    = {2023}
}

@misc{chen2026a,
  author       = {Chen, Ziang and Huang, Liqiang and Yang, Mengxuan and Zhou, Shengxuan},
  title        = {Regularity of Second-Order Elliptic PDEs in Spectral {Barron} Spaces},
  howpublished = {arXiv preprint},
  note         = {arXiv:2602.19381},
  year         = {2026}
}

@misc{choi2026,
  author       = {Choi, Jae-Hwan and Lim, Hyojae and Seo, Jinsol and Sim, Young-Jin and Song, Changhoon},
  title        = {Neural Network Approximation of Solutions to Fractional Parabolic Partial Differential Equations},
  howpublished = {arXiv preprint},
  note         = {arXiv:2607.27781},
  year         = {2026}
}

@book{devore1993,
  author    = {DeVore, Ronald A. and Lorentz, George G.},
  title     = {Constructive Approximation},
  publisher = {Springer},
  address   = {Berlin},
  year      = {1993}
}

@inproceedings{e2021barron,
  author    = {E, Weinan and Wojtowytsch, Stephan},
  title     = {Some observations on high-dimensional partial differential equations with {Barron} data},
  booktitle = {Proceedings of the 2nd Mathematical and Scientific Machine Learning Conference (MSML 2021)},
  series    = {Proceedings of Machine Learning Research},
  volume    = {145},
  pages     = {253--269},
  publisher = {PMLR},
  year      = {2022}
}

@book{eidelman1998,
  author    = {Eidelman, Samuil D. and Zhitarashu, Nicolae V.},
  title     = {Parabolic Boundary Value Problems},
  series    = {Operator Theory: Advances and Applications},
  volume    = {101},
  publisher = {Birkh\"auser},
  address   = {Basel},
  year      = {1998}
}

@article{fendler2006,
  author  = {Fendler, Gero and Gr\"ochenig, Karlheinz and Leinert, Michael},
  title   = {Symmetry of weighted $L^1$-algebras and the {GRS}-condition},
  journal = {Bulletin of the London Mathematical Society},
  volume  = {38},
  number  = {4},
  pages   = {625--635},
  year    = {2006},
  doi     = {10.1112/S0024609306018777}
}

@article{feng2026,
  author  = {Feng, Ye and Lu, Jianfeng},
  title   = {Solution theory of {Hamilton--Jacobi--Bellman} equations in spectral {Barron} spaces},
  journal = {SIAM Journal on Mathematical Analysis},
  volume  = {58},
  number  = {1},
  pages   = {636--660},
  year    = {2026},
  note    = {arXiv:2503.18656},
  doi     = {10.1137/25M1745763}
}

@book{friedman1964,
  author    = {Friedman, Avner},
  title     = {Partial Differential Equations of Parabolic Type},
  publisher = {Prentice-Hall},
  address   = {Englewood Cliffs, NJ},
  year      = {1964}
}

@article{han2018,
  author  = {Han, Jiequn and Jentzen, Arnulf and E, Weinan},
  title   = {Solving high-dimensional partial differential equations using deep learning},
  journal = {Proceedings of the National Academy of Sciences},
  volume  = {115},
  number  = {34},
  pages   = {8505--8510},
  year    = {2018},
  doi     = {10.1073/pnas.1718942115}
}

@book{henry1981,
  author    = {Henry, Daniel},
  title     = {Geometric Theory of Semilinear Parabolic Equations},
  series    = {Lecture Notes in Mathematics},
  volume    = {840},
  publisher = {Springer},
  address   = {Berlin},
  year      = {1981}
}

@article{kerimkulov2020,
  author  = {Kerimkulov, Bekzhan and \v{S}i\v{s}ka, David and Szpruch, {\L}ukasz},
  title   = {Exponential convergence and stability of {Howard's} policy improvement algorithm for controlled diffusions},
  journal = {SIAM Journal on Control and Optimization},
  volume  = {58},
  number  = {3},
  pages   = {1314--1340},
  year    = {2020},
  doi     = {10.1137/19M1236758}
}

@article{levi1907,
  author  = {Levi, Eugenio Elia},
  title   = {Sulle equazioni lineari totalmente ellittiche alle derivate parziali},
  journal = {Rendiconti del Circolo Matematico di Palermo},
  volume  = {24},
  pages   = {275--317},
  year    = {1907},
  doi     = {10.1007/BF03015067}
}

@book{rudin1962,
  author    = {Rudin, Walter},
  title     = {Fourier Analysis on Groups},
  publisher = {Interscience Publishers},
  address   = {New York},
  year      = {1962}
}

@book{yong1999,
  author    = {Yong, Jiongmin and Zhou, Xun Yu},
  title     = {Stochastic Controls: Hamiltonian Systems and {HJB} Equations},
  publisher = {Springer},
  address   = {New York},
  year      = {1999}
}

@book{rudin1991,
  author    = {Rudin, Walter},
  title     = {Functional Analysis},
  edition   = {2nd},
  publisher = {McGraw-Hill},
  address   = {New York},
  year      = {1991}
}

@book{taylor2023,
  author    = {Taylor, Michael E.},
  title     = {Partial Differential Equations {II}: Qualitative Studies of Linear Equations},
  edition   = {3rd},
  series    = {Applied Mathematical Sciences},
  volume    = {116},
  publisher = {Springer},
  address   = {Cham},
  year      = {2023},
  doi       = {10.1007/978-3-031-33700-0}
}

@article{favorov2022,
  author = {S. Yu. Favorov},
  title = {Local versions of the {Wiener--L\'evy} theorem},
  journal = {Matematychni Studii},
  volume = {57},
  number = {1},
  pages = {45--52},
  year = {2022},
  doi = {10.30970/ms.57.1.45-52}
}

@article{raissi2019,
  author  = {Raissi, Maziar and Perdikaris, Paris and Karniadakis, George Em},
  title   = {Physics-informed neural networks: A deep learning framework for solving forward and inverse problems involving nonlinear partial differential equations},
  journal = {Journal of Computational Physics},
  volume  = {378},
  pages   = {686--707},
  year    = {2019},
  doi     = {10.1016/j.jcp.2018.10.045}
}

@misc{choulli2026heat,
  author = {Mourad Choulli and Shuai Lu and Hiroshi Takase},
  title = {Heat Equations in Spectral Barron Spaces},
  year = {2026},
  note = {arXiv preprint arXiv:2608.26657}
}

\end{document}